\documentclass[11pt,letterpaper]{amsart}
\usepackage[left=3.5cm,right=3.5cm,top=2.2cm,bottom=2.2cm]{geometry} 
\usepackage{amsmath,amssymb,mathrsfs,stmaryrd}
\usepackage{amsthm}
\usepackage{mathtools,bm}
\usepackage{nicefrac} 
\usepackage[LGR,T1]{fontenc} 
\usepackage[utf8]{inputenc} %
\usepackage{lmodern}
\usepackage[english]{babel}
\usepackage{cancel}
\usepackage[colorlinks,linkcolor=NavyBlue,citecolor=Gray,urlcolor=Periwinkle]
{hyperref}
\usepackage[nameinlink,capitalize]{cleveref}
\usepackage[dvipsnames]{xcolor}
\usepackage[shortlabels] {enumitem}
\usepackage[babel=true,kerning=true]{microtype}
\usepackage[abbrev,msc-links,alphabetic]{amsrefs}
\usepackage{lscape}
\usepackage{afterpage}
\usepackage{lineno}
\usepackage[textsize=tiny]{todonotes}
\usepackage{slashed}
\usepackage{cancel}
\usepackage{empheq}

\usepackage{graphicx}
\makeatletter
\newcommand*\bigcdot{{\mathpalette\bigcdot@{.5}}}
\newcommand*\comp{{\mathpalette\bigcdot@{.65}}}
\newcommand*\bigcdot@[2]{\mathbin{\vcenter{\hbox{\scalebox{#2}{$\m@th#1\bullet$}}}}}
\newcommand{\mathds}{}
\makeatother
\DeclareFontEncoding{FMS}{}{}
\DeclareFontSubstitution{FMS}{futm}{m}{n}
\DeclareFontEncoding{FMX}{}{}
\DeclareFontSubstitution{FMX}{futm}{m}{n}
\DeclareSymbolFont{fouriersymbols}{FMS}{futm}{m}{n}
\DeclareSymbolFont{fourierlargesymbols}{FMX}{futm}{m}{n}
\DeclareMathDelimiter{\vvert}{\mathord}{fouriersymbols}{152}{fourierlargesymbols}{147}
\DeclarePairedDelimiter{\nn}{\vvert}{\vvert}

\colorlet{darkred}{red!90!black}

\theoremstyle{plain}
\newtheorem{theorem}{Theorem}[section]
\newtheorem*{theorem*}{Theorem}

\newtheorem{corollary}[theorem]{Corollary}

\newtheorem{lemma}[theorem]{Lemma}
\newtheorem{proposition}[theorem]{Proposition}

\theoremstyle{definition}
\newtheorem{definition}[theorem]{Definition}
\newtheorem{assumption}[theorem]{Assumption}

\newtheorem{example}[theorem]{Example}

\newtheorem{notation}[theorem]{Notation}

\newtheorem{remark}[theorem]{Remark}

\numberwithin{equation}{section}
\DeclarePairedDelimiter{\bk}{\llbracket}{\rrbracket}
\DeclareMathOperator\E{{\mathbb E}}
\DeclareMathOperator{\law}{Law}

\newcommand{\assign}{\vcentcolon=}

\newcommand{\tmmathbf}[1]{\ensuremath{\boldsymbol{#1}}}
\newcommand{\tmop}[1]{\ensuremath{\operatorname{#1}}}

\newcommand{\tand}{\quad\textrm{and} \quad}

\newcommand{\cff}{\mathcal F}
\newcommand{\LL}{\mathcal L}
\newcommand{\cll}{\mathcal L}
\def\cgg{{\mathcal G}}

\newcommand{\cpp}{{\mathcal P}}

\newcommand{\les}{\lesssim}
\newcommand{\X}{{\ensuremath{\bm{X}}}}
\newcommand{\XX}{\ensuremath{\mathbb{X}}}
\newcommand{\R}{\ensuremath{{\mathbf R}}}

\renewcommand{\P}{\ensuremath{{\mathbb P}}}

\newcommand{\DL}[3]{\ensuremath{\mathbf{D}_X^{#1}L_{#2,#3}}}

\newcommand{\Vone}{{\bar{V}}}

\newcommand{\abx}{{\mathcal X}}
\newcommand{\aby}{{\mathcal Y}}

\newcommand{\bxx}{\mathbb{X}}

\newcommand{\bmxx}{\bm{X}}
\newcommand{\C}{{\mathcal C}}
\newcommand{\CC}{{\mathscr{C}}}

\DeclarePairedDelimiter{\abs}{\lvert}{\rvert}
\newcommand{\EE}{{{\mathbb E}_{\bigcdot}}}
\newcommand{\W}{{\bm{\Omega}}}

\renewcommand{\tilde}{\widetilde}
\newcommand{\Bor}{\ensuremath{\mathscr B}}

\begin{document}
\title[McKean--Vlasov equations with rough common noise]{McKean--Vlasov equations with rough common noise}
\author{Peter K.~Friz}
\address{TU Berlin and WIAS Berlin}
\email{friz@math.tu-berlin.de}

\author{Antoine Hocquet}
\address{TU Berlin}
\email{hocquet@math.tu-berlin.de}

\author{Khoa Lê}
\address{School of Mathematics, University of Leeds, U.K.}
\email{k.le@leeds.ac.uk}

\subjclass[2020]{Primary 60H10, 60L20}

\keywords{McKean--Vlasov equation; Common noise.}

\begin{abstract} 
		We show well-posedness for McKean--Vlasov equations with rough common noise and progressively measurable coefficients.
	Our results are valid under natural regularity assumptions on the coefficients, in agreement with the respective requirements of It\^o and rough path theory. 
	To achieve these goals, we work in the framework of rough stochastic differential equations recently developed by the authors of this article. 
\end{abstract}
\maketitle

\tableofcontents
\section{Introduction} 

\label{sec:introduction}
Let $W,V,\Vone$ denote finite-dimensional real vector spaces
and consider the following McKean--Vlasov equation
\begin{equation}\label{eqn.MV}
	\left\{
	\begin{aligned}
		&dY_t(\omega)=b_t(\omega,Y_t(\omega),\mu_t)dt+\sigma_t(\omega,Y_t(\omega),\mu_t)dB_t(\omega)+(f_t,f'_t)(\omega,Y_t(\omega),\mu_t)d\X_t,
		\\& \mu_t=\law(Y_t),
	\end{aligned}
	\right.
\end{equation}
with initial condition \(Y_0(\omega)=\xi(\omega)\).
Here, $B$ is a $\{\cff_t\}$-standard Brownian motion in $\Vone$ defined on a suitable 
probability space $\W=(\Omega,\cgg,\{\cff_t\},\P)$, $\X=(X,\XX)$ is a (deterministic) rough path in $\CC^\alpha(V)$, $\alpha\in(\nicefrac13,\nicefrac12]$ and $\xi$ is a $\cff_0$-random variable in $W$. We  work with progressive coefficients
\begin{align*}
    (b,\sigma,f,f'):\Omega\times\R_+\times W\times\cpp_q(W)\to W\times \LL(\Vone,W)\times \LL(V,W)\times \LL(V,\LL(V,W)),
\end{align*}
 where $q\in[0,\infty)$ and $\cpp_q(W)$  is the space of probability measures $\mu$ on $W$ such that $\int_W\abs{x}^q \mu(dx)$ is finite.

The purpose of the current article is to show the following result, whose precise formulation will be given later.
\begin{theorem*}
	Assuming $\law(\xi)\in\cpp_q(W)$ and natural regularity conditions on the coefficients (see \cref{assume.regularity}), equation \eqref{eqn.MV} has a unique strong solution (see \cref{thm.MV.wellposed}) which depends continuously on the data $(\xi,b,\sigma,f,f',\X)$ (see \cref{thm.MKV.Stability}). 
\end{theorem*}

For a standard treatment of McKean--Vlasov dynamics (without common rough noise, $f \equiv 0$), classical references include \cite{MR1108185,MR1431299}. Common independent Brownian noise
 (think: $d\X \rightsquigarrow dW$) appears in \cite{CF16}, with $\sigma \equiv 0, f = f(y)$, and also \cite{MR3753660,coghi2019stochastic}, with general $\sigma = \sigma(y,\mu), f = f(y,\mu)$. Such stochastic common noise can always be reintroduced upon subsequent ``Brownian randomization'', i.e. $\X = (W, \int W \otimes dW)$ in \eqref{eqn.MV} (see Section \ref{sec:FR2SCN}).

Several authors went on to drop the assumption of Brownian noise and consider random RDEs, with given distribution of $\X = \X (\omega)$ on suitable rough path spaces, 
directly injected into the mean field dynamics. This started with \cite{MR3299600, MR3420480} in case of mean-field interaction in the drift vector field, followed by the 
remarkable extension \cite{MR4073682,bailleul2020propagation} to full mean-field dependence. 

The current article differs from the previous works on at least two accounts. One\footnote{ ... essentially the remark that stochastic integrals can only be written as randomized rough integrals under rigid structural assumption on the integrand that rule out progressive path dependence ... } 
is our ability to allow for progressively measurable dependence in our coefficients ($b, \sigma$). Any perspective on using rough paths in the context of controlled 
McKean--Vlasov dynamics \cite{carmona2015forward,lacker2017limit} and then mean field games with common noise (\cite{MR3753660} and references therein) will require such flexibility.  We also note that our approach leads to minimal regularity on the coefficient fields, notably Lipschitz regularity in $b, \sigma$, which can not be obtained with 
(random) rough path techniques as employed in \cite{MR3299600, MR3420480, MR4073682}.
We also mention \cite{coghi2019rough} which formally accommodates $\eqref{eqn.MV}$.
A close look reveals again somewhat restrictive regularity assumption (essentially, dimension dependent Sobolev regularity to embed in $C^3$) and linear dependence of $f$ with respect to $\mu$.

The other difference is that the measure arguments of the coefficients of \eqref{eqn.MV} have domain $\cpp_q(W)$ for some  $q\in[0,\infty)$ instead of $\cpp_2(W)$ as usually considered in the literature (e.g.  \cite{MR3752669} or \cite{MR4073682}).  
To treat the cases $\cpp_q$ (at least when $q\ge1)$, one could potentially work with coefficients whose regularity is described by means of \( L_q \)-Fr\'echet differentiability of the corresponding Lions lifts.
However, higher order differentiability in this sense is very restrictive (also in case $q=2$)  and rules out many natural functions on $\cpp_q$.
For the sake of argument, consider the degenerate McKean--Vlasov RSDE 
 \begin{equation}
 \label{ansatz_example}
\left \{\begin{aligned}
&dY_t=f(\mu_t)d\X_t,
\\
&\mu_t=\law(Y_t),\qquad 
\quad Y_0=\xi \in  L_1(\Omega)\,,
\end{aligned} 
\right .
 \end{equation}
 where \( W= \R\) and \( f(\mu)=\int_{W} h (y)\mu(dy)\), for some smooth compactly supported $h$. Introduce the Lions lift  \( \hat f(\eta) = f (\law(\eta)) \), the vector-valued path \( t\mapsto \eta_t=Y_t(\cdot) \) and regard
  \eqref{ansatz_example} as RDE
 \[
  d\eta_t=\hat f(\eta_t)d\X_t,\quad \quad \eta_0=\xi(\cdot) \in E:=L_1(\Omega)\,,
 \]
  on the infinite-dimensional state space \( E\), potentially addressable via RDE theory in Banach spaces.
  As is well-known, this would require \( \hat f \colon E\to \R\) to be \( \gamma \)-Lipschitz for some \( \gamma >\frac1\alpha\ge2\), hence in particular twice continuously differentiable
  which however is not a reasonable assumption, cf. the example in \cite[Remark 5.80]{MR3752669}.  

To overcome this problem, we characterize higher regularity using Fr\'echet differentiability not over the whole domain but only along a suitable sub-Banach spaces of random variables;
see \cref{def.Lionslift} (and also \cref{ex.lip} that relates to the above example). In this way, higher derivatives can be defined and utilized in our analysis. 
To the best of our knowledge this perspective is new in the (vast) literature on differential calculus on Wasserstein spaces. 
 Besides the obvious advantage that $q$ can be arbitrary in $[0,\infty)$,
we find that Fr\'echet differentiability (along subspaces) yields more familiar higher order Taylor expansions, whereas Lions derivatives of orders more than $2$ can be algebraically involved because each Lions derivative carries an additional argument (but see \cite{delarue2021probabilistic, delarue2022probabilisticroughpathsii} how to manage the resulting complexity).  That said, our results do accommodate the general class of coefficients considered in \cite{MR4073682} as is verified in \cref{ex.reg2}.
Having at our disposal a suitable notion of regularity of arbitrary order with respect to the measure argument, the rough integration against $\X$ in \eqref{eqn.MV} can be understood as a rough stochastic integral of a stochastic controlled rough path, as defined in the earlier work \cite{FHL21}. This point of view is independent from and is complementary to the approach of probabilistic rough paths introduced in \cite{MR4073682}. 

Our approach is based upon the analysis of rough stochastic differential equations (RSDEs) developed in \cite{FHL21}. This method takes advantage of the stability of rough stochastic differential equations while treating the measure dependent coefficients as stochastic controlled vector fields.
An alternative approach which originates from our analysis is to replace the vector fields $b,\sigma,f,f'$ by their corresponding Lions lifts and then treat \eqref{eqn.MV} directly as a rough stochastic differential equation with coefficients having their domains on $\Omega\times \R_+\times W\times L_q(\Omega,W)$, see \cref{rmk.2ndway} below. 
However, we have chosen the former approach to discuss in details because it is, in essence, analogous to the classical treatment of \cite{MR1108185}  with It\^o theory being replaced by the one from \cite{FHL21}.
For simplicity, we only consider \eqref{eqn.MV} with uniformly bounded coefficients. McKean--Vlasov equations with unbounded (smooth) coefficients are not well-posed in general, with counter examples given in \cite{scheutzow1987uniqueness}. We  shall not discuss equations with unbounded coefficients herein but refer to \cite{MR4260494} for recent affirmative developments.


\medskip
{\bf Outline of the paper.} The remainder of the present section introduces our main notations (even though some additional ones will be introduced along the text) and discusses a useful concept of regularity along certain specified directions. This concept is essential in characterizing the regularity of the coefficients with respect to the measure components (see \cref{ex.reg2,ex.lip}). We also recall standard notions for (controlled) rough paths. In \cref{sec:preliminaries}, we introduce the notions of stochastic controlled rough paths/vector fields and state our main results. \cref{sec.preparatory_material} is devoted to preliminary results that will be used throughout the rest of the paper. In particular, some properties of Lions lifts are discussed, along with a mean-value theorem for measure-dependent vector fields which is crucial in our analysis. We then state and prove the main composition lemmas which contain estimates that are used all throughout the paper. These form a central tool describing the construction of \( W \)-dependent vector fields from measure-dependent ones, composed with certain families of random variables.
The main results are proved in \cref{sec:proofs}.

\medskip

\noindent {\bf Funding:}  PKF acknowledges support from DFG CRC/TRR 388 ``Rough
Analysis, Stochastic Dynamics and Related Fields'', Projects A02, A07 and B04. 
KL acknowledges support from EPSRC
[grant number EP/Y016955/1]. PKF and KL thank Fabio Bugini for many discussions.

\section{Notation and preliminaries}
\subsection{Frequently used notation} 
\label{sub.rough_paths}
The notation $F\lesssim G$ means that $F\le CG$ for some positive constant $C$ which depends only on some universal parameters. We set $0/0 := 0$.
Throughout the manuscript we denote by $\Bor(E)$ the Borel-algebra of a topological space $E$.
We denote generic (i.e.\ without dimensionality assumption), real Banach spaces by the symbols $(\abx,|\cdot|_{\abx}),$ $(\mathcal Y,|\cdot|_{\mathcal Y})$ or simply \( \abx,\mathcal Y \). 
The notations \( V,\Vone,W,\bar W\) always refer to finite-dimensional real vector spaces, 
for which the norms are denoted indistinctly by \( |\cdot| \). 
For each integer $k\ge1$, $\LL^k(\abx,\mathcal Y)$ denotes the space of bounded multilinear maps from $[\abx]^k$ to $\mathcal Y$ (endowed with the usual induced norm) and \( \mathcal L:=\mathcal L^1\).
Let \( E\) be a Polish vector space, i.e.\ a separable topological vector space which is metrizable and fix a metric \( d\colon E\times E\to[0,\infty) \). 
We will let $\cpp(E)=\cpp_0(E)$ be the set of Borel probability measures on $E$ and $\cpp_q(E)$ be the subset of $\cpp(E)$ containing measures $\mu$ of order \( q >0\), i.e.\ such that $\int_E d(x_0,x)^q\mu(dx)<\infty$ for some (hence every) \( x_0\in E \).

Troughout the manuscript we assume that \( \W= (\Omega,\mathcal G,\P)\) is a Polish probability space. Integration against \( \P \) will be denoted by \( \E \) and conditional expectation with respect to a sub sigma-algebra \( \mathcal F\subset\mathcal G \) will be denoted by \( \E(\cdot|\mathcal F) \) or \( \E_{\mathcal F}  \).
By a \textit{\( \mathcal Y \)-valued random variable} we mean a map \( Z\colon \Omega\to \mathcal Y \) which is  \textit{strongly measurable}, in the sense that it is \( \mathcal G /\Bor(\mathcal Y)\)-measurable and separably-valued.%
\footnote{Meaning that there is a closed, separable subspace \( \mathcal Z \subset \mathcal Y\) such that \( Z \) is supported in \( \mathcal Z \).}
For a real number \( q\ge0 \),
the space of $\mathcal Y$-valued, $q$-integrable (i.e.\ with law in \( \mathcal P_q(\mathcal Y) \)) random variables is denoted by $L_q(\mathcal Y)$ or $L_q(\cgg;\mathcal Y)$ in order to emphasize the underlying sigma-algebra.
We denote by
\begin{equation}
	\label{lebesgue_q}
	Z\mapsto\|Z\|_{q}
	=
	\begin{cases}
		\E[|Z|_{\mathcal Y}^q ]^{\frac1q\wedge 1}\quad \text{if}\enskip q>0
		\\
		\E[|Z|_{\mathcal Y}\wedge1] \quad \text{if }q=0.
	\end{cases}
\end{equation}
The space \( (L_q,\|\cdot\|_{q}) \) is Banach if \( q\ge1 \) (it is the usual space of \( q \)-Lebesgue integrable random variables). It is a completely metrizable vector space if \( q\in [0,1) \) (endowed with the distance \( Z,\bar Z\mapsto \|Z-\bar Z\|_q \)). Some useful properties are recalled in \cref{app:Lq}.

\subsection{Lipschitz regularity along sub-Banach spaces}
\label{sec:along}

Fix a Banach space \( \aby \) and a topological vector space \( E \). 
For a map \( g\colon E\to \aby \) we denote the usual supremum norm by
\[
|g|_{\infty;E}=\sup_{x\in E}|g(x)|_\aby
\]
and we will simply write \( |g|_{\infty} \) if \( E \) is clear from the context.
The Banach space of strongly \( \Bor(E)/\Bor(\aby) \)-measurable maps so that \( |g|_\infty<\infty\) will be denoted by \( \mathcal B_b(E,\aby)\).\smallskip

Suppose that a Banach space $\abx$ exists such that $\abx\hookrightarrow E$, in sense of continuous embedding.
The next definition is a generalization of usual notions of regularity when the set of directions needs to be restricted to \( \abx \).
\begin{definition}\label{def.Xfrechet}\phantom{new line}
\begin{enumerate}[label=(\roman*)]
\item For $\kappa\in(0,1]$, $g$ is called $\kappa$-H\"older along $\abx$ if
\begin{equation}
\label{holder_along}
[g]_{\kappa;\mathcal X\subset E}= \sup_{x,y\in E:\ x-y\in \mathcal X}\frac{|g(x)-g(y)|_\aby}{|x-y|_\abx^\kappa} <\infty\,.  
\end{equation}

\item \label{def.Frechet} A function $g\colon E\to\aby$ is \( k \)-times Fr\'echet differentiable \textit{along $\abx$} 
(or \( k \)-times $\abx$-Fr\'echet differentiable)
if for every $x\in E$, the map \( \abx\to\aby \), \( \zeta\mapsto g(x+\zeta) \) is \( k \)-times Fréchet differentiable.
We denote by $D^kg(x)$ the corresponding derivative at $\zeta=0$ and we call it
the \( k \)-th Fréchet derivative of $g$ at $x$ along \( \mathcal X \) (obviously this depends on the choice of sub-Banach space, but we do not emphasize \( \abx \) in the corresponding symbol).

\item 
Similarly, a function \( g\colon E\to\aby \) is called \( k \)-times \textit{continuously} Fr\'echet differentiable along $\abx$ provided that it is \( k \)-times Fr\'echet differentiable along $\abx$ and that  the map $x\mapsto D^kg(x)$ is  continuous from $E$ to $\LL^k(\abx,\aby)$.

\item \label{def.Lip} For any real number of the form \( \gamma =N +\kappa\) with \( N=0,1,2,\dots \) and \( \kappa\in (0,1] \), a function $g\colon E\to\aby$ is called $\gamma$-Lipschitz along $\abx$ if it is \(N\)-times Fr\'echet differentiable along \(\mathcal X\) and its \( N \)-th Fr\'echet derivative along \( \mathcal X \) is itself $\kappa$-H\"older along $\abx$.
\end{enumerate}
We will denote by $\C^\gamma_{b,\abx}(E,\aby)$  the space of all $\gamma$-Lipschitz functions along $\abx$. When \( E =\abx\), \cref{def.Xfrechet} deduces to the usual Lipschitz function spaces, so we will drop the unnecessary subscripts and write \( \C^\gamma_{b}(\abx,\aby):=\C^\gamma_{b,\abx}(\abx,\aby) \).
\end{definition}
For each $g\in\C^\gamma_{b,\abx}(E,\aby)$, we define
\begin{equation}\label{def.hdernorm}
	[g]_{\gamma;\abx\subset E} :=\sum\nolimits_{k=1}^N|D^kg|_{\infty;E}+[D^Ng]_{\kappa;\abx\subset E}, \quad 
	|g|_{\gamma;\abx\subset E} :=|g|_{\infty;E}+[g]_{\gamma;\abx\subset E}\,.
\end{equation}
Again, if \( E=\abx \), we will drop the corresponding subscript in the notations.
The next statement is a mean-value theorem for functions on \( E \) which are differentiable along \( \abx \).

\begin{theorem}[Mean-value theorem]\label{thm.mvt}
Let \( E \) be a topological vector space and consider real Banach spaces
	\( \abx\hookrightarrow E\) and \(\aby \) a real Banach spaces as before. 
	Suppose that $g\colon E\to\aby$ is \textit{continuously} Fr\'echet differentiable along $\abx$. 
		Then, for any $x \in E$ and $\zeta \in \abx$, 
	the map \( \theta\mapsto Dg(x+\theta \zeta)[\zeta] \) from \( [0,1]\to \aby \) is strongly continuous  
	and we have 
	\begin{align}
		g(x+\zeta)-g(x)
	=\int_0^1 Dg(x+ \theta \zeta)[\zeta]d\theta\,.
	\label{id.mvt-intro}
	\end{align}
	Also, if  \( |Dg|_{\infty;E} 
	<\infty\), then \( g \) is Lipschitz along $\abx$ and \( [g]_{1;\abx\subset E}\le |Dg|_{\infty;E} \).
\end{theorem} 

\begin{proof} 
Fix $x \in E$. By \cref{def.Xfrechet}, the map $g_x : \zeta \mapsto g (x + \zeta)$ is continuously differentiable along $\abx$.
Furthermore, $D g (x) = D g_x (0) \in \LL (\abx, \aby)$ and the map $e \mapsto D g (e) \in \mathcal{L} (\abx, \aby)$ is continuous; and so is
$\theta \mapsto  D g (x + \theta \zeta) \in \cll (\abx, \aby)$ and thus
also $\theta \mapsto 
D g (x + \theta \zeta)[\zeta] \in \aby
$.
Since
 $
 D g (x + \theta \zeta)[\zeta]
 = 
 D g_{x + \theta \zeta}(0)[\zeta]
  = \frac{\partial}{\partial \theta} g_{x + \theta \zeta} (0) =
\frac{\partial}{\partial \theta} g_x (\theta \zeta)$ the fundamental theorem of calculus for $C^1$ paths in Banach spaces yields\footnote{Integrals here are understood in Riemann sense.} 
\[ g (x + \zeta) - g (x) = g_x (\zeta) - g_x (0) = \int_0^1 \frac{\partial}{\partial
   \theta} g_x (\theta \zeta) d \theta = 
   \int_0^1 D g (x + \theta \zeta)[\zeta]
   d \theta .
    \]
It also follows that $[g]_{1 ; \abx \subset E} \le |Dg|_{\infty ; E}$, since
\[ | D g (x + \theta \zeta) [\zeta] | \leqslant |Dg (x + \theta \zeta) |_{\LL (\abx, \aby)} \times | \zeta |_{\abx}
   \leqslant |Dg|_{\infty ; E} \times | \zeta |_{\abx} .\]
\end{proof}  

The following example will be important for our purposes.
\begin{example}[Regularity along Banach spaces of random variables]
\label{exa:regularity_rv}
Let \( \W = (\Omega,\mathcal G,\mathbb P) \) be a probability space and consider moment spaces 
\( L_r =L_r(\W)\) for \( r\ge0 \). Fix real numbers $\gamma\ge0$, $q \ge0$ and $p\ge 1\vee q$
and take \( g \) to be an element in the space of \( \gamma \)-Lipschitz functions on \( L_q \) along \( L_p\subset L_q \). We use the shorthand notation, with $\mathcal{Y}$ a real Banach spaces as before,
\begin{equation}
\label{nota:Cbpq}
	\begin{aligned}
&g\in\C^\gamma_{b,p,q}:=\C^\gamma_{b,L_p}( L_q,\mathcal Y) 
\\
&|g|_{\mathcal C^\gamma_{b,p,q}}:=|g|_{\infty; L_q}+[g]_{\gamma;L_p\subset L_q}\,.
\end{aligned}
\end{equation}
By definition, for each integer $0\le k<\gamma$ and each $\xi\in  L_q$, $D^kg(\xi)$ is a multilinear map on $[L_p]^k$ and 
\[
	\sup_{\xi\in  L_q}|D^kg(\xi)[\eta_1,\ldots,\eta_k]|\le 
	|g|_{\mathcal C^\gamma_{b,p,q}}
	\|\eta_1\|_p\ldots\|\eta_k\|_p \quad \text{for all } \eta_1,\ldots,\eta_k\in L_p.
\]
Furthermore, if $\gamma=N+ \kappa$ for some integer \(N\) and \(\kappa\in (0,1]\), then we have 
\[
	|D^{N}g(\xi)[\eta_1,\ldots,\eta_N]-D^{N}g(\bar \xi)[\eta_1,\ldots,\eta_N]|
	\le 
	|g|_{\mathcal C^\gamma_{b,p,q}}
	\|\xi-\bar \xi\|_p^{\kappa}\|\eta_1\|_p\ldots\|\eta_{N}\|_p
\]
for every $\xi,\bar \xi\in  L_q$; $ \eta_1,\ldots,\eta_{N}\in L_p$.


\begin{remark}
Similar notations will be used if we substitute \( E= W\times L_q\) and \( \abx=W\times L_p\).
This will result in a slight abuse of notation, which however can be easily clarified from the context.
In particular, the previously discussed properties are equally true, mutatis mutandis, of
\[
\C^\gamma_{b,p,q}:=\C^\gamma_{b,W\times L_p}(W\times  L_q,\mathcal Y).
\]
Namely, if \( g\in \C^\gamma_{b,p,q} \) then for any integer \(k\in[0,\gamma)\),
\begin{equation}\label{est.normDkg}
		\sup_{(y,\xi)\in W\times L_q}|D^k g(y,\xi)[(z_1,\eta_1),\ldots,(z_k,\eta_k)]|\le 
		|g|_{\mathcal C^\gamma_{b,p,q}}
		(|z_1|+ \|\eta_1\|_p)\ldots(|z_k|+ \|\eta_k\|_p)
\end{equation}
and if $\gamma=N+ \kappa$ for some integer \(N\) and \(\kappa\in (0,1]\), then
\begin{multline}
\label{est.normdgammag}
		|D^{N}g(y,\xi)[(z_1,\eta_1),\ldots,(z_N,\eta_N)]-D^{N}g(\bar \xi,\bar y)[(z_1,\eta_1),\ldots,(z_N,\eta_N)]|
		\\
		\le |g|_{\C^\gamma_{b,p,q}}(|y-\bar y|+\|\xi-\bar \xi\|_p)^{\kappa}
		(|z_1|+\|\eta_1\|_p)\ldots(|z_N|+\|\eta_N\|_p)
\end{multline}
for every $y,\bar y,z_1,\dots z_N\in W$; $\xi,\bar \xi\in  L_q$ and \( \eta_1,\dots,\eta_N\in L_p. \)
These properties will be used tacitly in the rest of the paper. 
\end{remark}
\end{example}
\subsection{Rough paths}
Throughout the paper we fix a finite time horizon \( T>0\) and make use of the the following notation for paths. For every \( [c,d]\subset[0,T] \), we introduce
\[
\begin{aligned}
	&\Delta (c,d) \vcentcolon=\{(s,t)\in[0,T]^2:\enskip c\le s\leq t\le d\},
\end{aligned}
\]
and we abbreviate \( \Delta=\Delta(0,T)\). 
Let \( E \) be a topological vector space and let \( Y\colon [0,T]\to E\) be a continuous path. We introduce the increment of \( Y \) as the two-parameter quantity
\begin{equation}
	\label{nota:increments}
	\delta Y_{s,t}\vcentcolon=Y_t-Y_s,
\end{equation} 
for every \( (s,t)\in\Delta \).
If \( E=(\mathcal Y,|\cdot|) \) is Banach and \( \kappa\ge0 \), we denote by \( C^{\kappa}_2 ([0,T];\mathcal Y)\) the space of
\( 2 \)-index maps \(A\colon\Delta\to \mathcal Y\) such that
\begin{equation}
\label{holder_A}
  |A|_\kappa \vcentcolon=\sup_{(s,t)\in\Delta,s\neq t }\frac{|A_{s,t}|}{(t-s)^\kappa } <\infty 
\end{equation}
 which is also Banach.
Consistently with \cref{sec:along}, we also let \( C^\kappa([0,T];\mathcal Y)=\{Y\colon [0,T]\to\mathcal Y\text{ such that } \delta Y\in C_2^{\kappa}([0,T];\mathcal Y)\}\). It is equipped with the norm \( |Y|_\kappa:=\sup_{t\in [0,T]}|Y_t| + |\delta Y|_{\kappa} \).

Finally, we recall the definition of a 2-step, H\"older rough path (as encountered for instance in \cite{friz2020rough}). Recall that \( V\) refers to a finite-dimensional real Banach space.
\begin{definition}\label{def.RP}
	Let $\alpha$ be a fixed number in  $(\nicefrac13,\nicefrac12]$ and $T>0$. We call $\X=(X,\XX)$ a two-step, $\alpha $-H\"older \textit{rough path} on $[0,T]$ with values in $V$, in symbols 
	$\X \in \mathscr{C}^\alpha ([0,T];V)$, if and only if
	\begin{enumerate}[(i)]
		\item $(X,\XX)$ belongs to $C^\alpha([0,T];V)\times C_2^{2 \alpha}([0,T];V\otimes V)$,
		\item for every $(s,u,t)\in [0,T]^3$ with \( s\le u\le t \), Chen's relation 
		\begin{equation}\label{chen}
			\bxx_{s,t}-\bxx_{s,u}-\bxx_{u,t}=\delta X_{s,u}\otimes \delta X_{u,t}
		\end{equation}
is satisfied.
	\end{enumerate} 
For $\X$ as above, we denote
\[
|\X|_\alpha=|\delta X|_\alpha+|\XX|_{2 \alpha}
\tand\nn{\X} _\alpha=|\delta X|_\alpha+|\XX|_{2 \alpha}^{1/2}.
\]	
For any \( \alpha'\in[0,1]\), the distance $\rho_{\alpha,\alpha'}$ between two rough paths $\X$ and $\bar\X$ is defined by
\begin{equation}
	\label{def.rho_metric}
	\rho_{\alpha,\alpha'}(\X,\bar\X)=|\delta X- \delta\bar X|_\alpha+|\XX-\bar\XX|_{\alpha+ \alpha'}.
\end{equation}
We abbreviate $\rho_\alpha=\rho_{\alpha,\alpha}$.
\end{definition}
%


\subsection{Rough stochastic differential equations} 
\label{sub:rsde}
Let $\W=(\Omega,\mathcal G,\{\cff_t\},\P)$ be a filtered probability space, $B$ be a standard $\{\cff_t\}$-Brownian motion in $\Vone$, $\bmxx=(X,\bxx) $ be a deterministic rough path in $\mathscr{C}^\alpha(V)$  with $\alpha\in(\frac13,\frac12]$. We consider the rough stochastic differential equation
\begin{equation}\label{eq:RSDE}
	dY_t(\omega)=b_t(\omega,Y_t(\omega))dt+\sigma_t(\omega,Y_t(\omega))dB_t(\omega)+(f_t,f'_t)(\omega,Y_t(\omega))d\X_t
	,\quad t\in[0,T].
\end{equation}
We are given a drift nonlinearity
$b\colon \Omega\times [0,T]\times W\to W$, and vector fields
$\sigma\colon\Omega\times[0,T]\times W\to \LL(\Vone,W)$, 
$f\colon \Omega\times[0,T]\times W\to\LL(V,W)$, $f'\colon \Omega\times[0,T]\times W\to \LL(V,\LL(V,W))$.
We assume further that $b,\sigma,f,f'$ are progressive and that for each $t$ and a.s. $\omega$, $y\mapsto f_t(\omega,y)$ is differentiable with derivative $Df_t(\omega,y)$.
In what follows, we omit the $\omega$-dependence in the coefficients $\sigma,b,f,f',Df$.
The definition of integrable solutions to rough stochastic differential equations as above is completely transparent from \cref{def.MVsoln} (just drop the \(  L_q'(W)\)-dependence in the variables). Hence, we do not recall it here and refer nonetheless to \cite{FHL21} for details.

We start with a statement which asserts that solutions of \eqref{eq:RSDE} satisfy a priori estimates and provides quantitative bounds in terms of \( \mathbf D_X^{\alpha,\beta}L_{m,\infty}\) semi-norms. For our purposes, we need to emphasize the dependency upon the best constant \( \varkappa=\varkappa(W)>0\) that appears in the upper BDG inequality for \( W\)-valued processes (recall that \( W\) is finite dimensional).
\begin{proposition}[{\cite[Proposition 4.5]{FHL21} }]\label{prop.apri}
	Suppose that $b,\sigma$ are random bounded continuous and $(f,f')$ belongs to $\mathbf{D}^{\beta,\beta'}_XL_{m,\infty}\C^{\gamma-1}_b$ with
	\(\beta \in (0,\alpha]\), \(\beta'\in(0,1]\) and \(\gamma\in (2,3]\) such that \(\alpha+(\gamma-1)\beta>1\) and \(\alpha+\beta+\beta'>1\).
	Let $Y$ be an $L_{m,\infty}$-solution to \eqref{eq:RSDE} 
	and take any finite constant $M$ such that
	\begin{equation*}
		\|(f,f')\| _{\gamma-1;\infty}+\bk{(f,f')}_{X;\beta,\beta';m,\infty}\le M.
	\end{equation*}
	Then, there exists a constant $C$ depending only on $T,m,\alpha,\beta,\beta',\gamma$ and \( \varkappa\) such that 
	\begin{gather}\label{est.apri.m}
		\|\delta  Y\|_{\alpha;m,\infty}+\|\E_\bigcdot R^Y\|_{\alpha+\beta;\infty}
		\le C(1+\|b\|_\infty+\|\sigma\|_\infty+M\nn{\X}_\alpha)^{1/\beta''},
	\end{gather}
	where $\beta''=[(\gamma-2)\beta]\wedge\beta'$ and $R^{Y}_{s,t}=\delta Y_{s,t}-f(Y_s)\delta X_{s,t}$.
\end{proposition}
In particular, the constant \( C\) in \eqref{est.apri.m} does not depend on the particular dimension of \( W\).
However, the dependency on $W$ still presents itself when estimating $\kappa(W)$ and the norms of the coefficients \( b,\sigma,f,f'\), whose definition depends upon the induced norms of \( \LL(\Vone,W),\LL(V,W),\LL(V,\LL(V,W))\).

The next result summarizes the main existence and uniqueness results for rough stochastic differential equations of this kind.
\begin{theorem}[{\cite[Theorem 4.6]{FHL21}}]\label{thm.fixpoint}
	Let $m\ge2$ be real and $\X\in \CC^\alpha$ with \( \frac13< \alpha\le\frac12 \). Let $b,\sigma$ be random bounded Lipschitz functions (see \cite[Definition 4.1]{FHL21}), assume that $(f,f')$ belongs to  $\mathbf{D}^{\alpha, \alpha}_XL_{m,\infty}\C^\gamma_b$ while $(Df,Df')$ belongs to  $\mathbf{D}^{\alpha,\alpha''}_XL_{m,\infty}\C^{\gamma-1}_b$.
	Assume moreover that $\gamma>\frac1 \alpha$ and $2 \alpha+\alpha''>1$. Then for every $\xi\in  L_0(\cff_0;W)$, there exists a unique $L_{m,\infty}$-integrable continuous  solution to \eqref{eq:RSDE} starting from $\xi$ over any finite time interval.
\end{theorem}	
Finally, we recall the main result concerning the continuous dependence of solutions of \eqref{eq:RSDE} with respect to data.
\begin{theorem}[{\cite[Theorem 4.9]{FHL21}}]
	\label{thm.stability_precise}
	Let $\xi,\bar \xi$ be in $ L_0(\cff_0)$; $\X,\bar \X$ be in $\CC^\alpha$, $\alpha\in(\frac13,\frac12]$; \( \sigma,\bar\sigma,b,\bar b \) be random bounded continuous functions; fix \( m\ge2 \), and parameters \( \gamma\in (2,3] ,\beta\in [0,\alpha] \) such that \(\alpha + (\gamma-1)\beta>1\).
	Consider
	$(f,f')\in \mathbf{D}^{\beta,\beta}_XL_{m,\infty}\C^\gamma_b$ such that $(Df,Df')$ belongs to $\mathbf{D}^{\beta,\beta'}_XL_{m,\infty}\C^{\gamma-1}_b$
	where \( \beta'>0 \) is taken such that
	\[
	1-\alpha-\beta <\beta'\le1,
	\]
	and fix another stochastic controlled vector field $(\bar f,\bar f')\in\mathbf{D}^{\beta,\beta'}_{\bar X}L_{m,\infty}\C^{\gamma-1}_b$.
	Let $Y$ be a $L_{m,\infty}$-integrable solution to \eqref{eq:RSDE} starting from $\xi$, and similarly denote by
	$\bar Y$ a $L_{m,\infty}$-integrable solution to \eqref{eq:RSDE} starting from $\bar \xi$ with associated coefficients $(\bar \sigma,\bar f,\bar f',\bar b,\bar \X)$.
	
	Then, denoting by \( R^Y_{s,t}=\delta Y_{s,t}-f(Y_s)\delta X_{s,t}$ and $\bar R^{\bar Y}_{s,t}=\delta\bar Y_{s,t}-\bar f(\bar Y_s)\delta\bar X_{s,t}\)
	and recalling the notations \eqref{def.rho_metric} and \eqref{def.bk_metric},
	we have the estimate 
	\begin{align}
		&\|\sup_{t\in[0,T]}|\delta Y_{0,t}-\delta\bar Y_{0,t}|\|_m+ \|\delta Y- \delta\bar Y\|_{\alpha;m}+\|\delta f(Y)-\delta\bar f(\bar Y)\|_{\beta;m} + \|\E_{\bigcdot} R^Y-\E_{\bigcdot}\bar R^{\bar Y}\|_{\alpha+ \beta;m}
		\nonumber\\&\lesssim\||\xi-\bar \xi|\wedge1\|_m+\sup_{t\in[0,T]}\|\sup_{x\in W}|\sigma_t(x)-\bar \sigma_t(x)|\|_m+\sup_{t\in[0,T]}\|\sup_{x\in W}|b_t(x)-\bar b_t(x)|\|_m
		\nonumber\\&\quad+\rho_{\alpha,\beta}(\X,\bar \X)+\|(f-\bar f,f'-\bar f')\|_{\gamma-1;m}+\bk{f,f';\bar f,\bar f'}_{X,\bar X;\beta,\beta';m},
		\label{est.stability_precise}
	\end{align}
	where the implied constant depends on $|\X|_\alpha,|\bar\X|_\alpha,b,\sigma,f,f'$ and \( \varkappa.\)
\end{theorem}
These results are far-reaching applications of the stochastic sewing lemma from \cite{le2018stochastic,LeSSL2}. In the current article, we will not apply the stochastic sewing lemma directly but rather rely solely on \cref{thm.fixpoint,thm.stability_precise}.

\section{Assumptions and main results} 
\label{sec:preliminaries}
Hereafter, we suppose that $\W=(\Omega,\mathcal G,\{\cff_t\}_{t\in [0,T]},\P)$ is a stochastic basis satisfying the usual conditions (in particular \( \Omega \) is complete and the sigma-algebra $\cff_0$ contains the $\P$-null sets).
Throughout the section we fix two real numbers \(m,q\) such that
\[
q \in [0,\infty)
,\quad \quad 
m\ge q\vee1\,.
\]

\subsection{Stochastic controlled processes and vector fields} 
\label{sub:stochastic_controlled_processes}

We introduce some useful classes of stochastic processes \( Z\colon \Omega\times \mathbb T \to \mathcal Y \) with index set
\( \mathbb T\in \{[0,T], \Delta (0,T)\}.\)
We assume tacitly that these processes are $\mathcal G\otimes\Bor(\mathbb T)/\Bor( \mathcal Y)$-measurable. A process
\( Z \) is called \( m\)-integrable if $Z_\tau$ is in $L_m$ for every $\tau\in \mathbb T$ and uniformly \(m \)-integrable if \( \sup_{\tau\in\mathbb T}\|Z_\tau\|_m <\infty.\)

\begin{notation}[Conditional type norms] 
	Let $\beta\ge0$ and fix an extended real number $n \in [ m, \infty]$.
	\begin{itemize}
		\item 
		If \( \mathcal F\subset \mathcal G\) is another \( \sigma\)-algebra and \( Z\colon \Omega\to \mathcal Y\) a random variable, we introduce 
		\[
		\|Z|\mathcal F\|_{m}:= (\mathbb E[|Z|^m_{\mathcal Y}|\mathcal F])^{1/m}\,.
		\]
		
		\item
		For \( \mathbb T=\Delta(0,T) \), we define
		\[
		\|Z\|_{\beta;m,n}\vcentcolon=
		\sup_{(s,t)\in \Delta(0,T)}\frac{\|\|Z_{s,t}|\cff_s \|_m\|_n}{(t-s)^\beta}
		\, \in[0,\infty]
		\]
		and write \( \|Z\|_{\beta;m}\vcentcolon=\|Z\|_{\beta;m,m}\).
		\item For \( \mathbb T=[0,T] \) we define
		\( \|Z\|_{\beta;m,n}\vcentcolon=
		\sup_{t\in [0,T]}\|Z_t\|_{n} + \|\delta Z\|_{\beta;m,n} \) and write \( \|Z\|_{\beta;m}:=\|Z\|_{\beta;m,m} \). 
	\end{itemize}
\end{notation}

We will also make use of the following conventions.
\begin{notation}[Shorthands for conditional expectations]\phantom{For spacing }
	\begin{itemize}
		\item
		Given \( s\in [0,T] \), we let
		\begin{equation*}
			\E_s = \E(\,\cdot\,|{\mathcal F_s})
		\end{equation*}
		\item 
		For a two-parameter family of random variables \( Z_{s,t}(\omega) , s\le t\in [0,T]\), we will denote by $\E_\bigcdot Z$ the map \( (\omega;s,t) \mapsto \E_s (Z_{s,t} )(\omega) \).
	\end{itemize}
\end{notation}

Next, we recall the notion of stochastic controlled processes from \cite{FHL21}.
Let $\bar X$ be a path in $C^\alpha([0,T];V)$ and recall that \( \mathcal Y \) is a generic Banach space (it needs not be separable).
\begin{definition}[Stochastic controlled rough paths]
	\label{def.regSCRP}
	Fix parameters $0\le\beta,\beta'\le1$ and $n\in[m,\infty]$.
	The pair $(Z,Z')$ is called a \textit{stochastic controlled rough path} (with values in \( \mathcal Y \)), on the stochastic basis $\W$,
	and we write
	\[
	(Z,Z')\in\DL{\beta,\beta'}{m}{n}(\mathcal Y)=\DL{\beta,\beta'}{m}{n}([0,T],\W;\mathcal Y)
	\]
	if the following conditions are satisfied
	\begin{enumerate}[(i)]
		\item\label{good1} $Z\colon \Omega\times [0,T]\to \mathcal Y$ and $Z'\colon\Omega \times[0,T]\to \LL(V,\mathcal Y)$
		are progressively $\{\mathcal F_t\}$-measurable in the strong sense;
		\item\label{good2} the quantities $\|\delta Z\|_{\beta;m,n}$, $\sup_{t\in[0,T]}\|Z'_t\|_n$ and $\|\delta Z'\|_{\beta';m,n}$ are finite;
		\item putting $R^Z_{s,t}
		=\delta Z_{s,t}-Z'_s \delta X_{s,t}$ for each $(s,t)\in \Delta$,
		we have that
		$\E_{\bigcdot} R^Z\equiv [(\omega;s,t)\mapsto \mathbb E_s R^Z_{s,t}]$ is such that $\|\EE R^Z\|_{\beta+\beta';n}<\infty$.
	\end{enumerate}
\end{definition}
 For each $(Z,Z')$ and $(\bar Z,\bar Z')$ respectively in $\mathbf{D}^{\beta,\beta'}_{X}L_{m,\infty}(\mathcal Y)$ and $\mathbf{D}^{\beta,\beta'}_{\bar X}L_{m,\infty}(\mathcal Y)$, we define the distance%
\footnote{This is slightly different from the one used in \cite{FHL21}.} 
\begin{align}
	\label{def.scrp.metricd}
	\| Z,Z';\bar Z,\bar Z'\|_{X,\bar X;\beta,\beta';m}
	=\| Z-  \bar Z\|_{\beta;m}
	+\| Z'- \bar Z'\|_{\beta';m}
	+\|\E_\bigcdot R^{Z}-\E_\bigcdot \bar R^{\bar Z}\|_{\beta+\beta';m},
\end{align}
where $R^Z_{s,t}\vcentcolon=\delta Z_{s,t}-Z'_s \delta X_{s,t}$ while $\bar R^{\bar Z}_{s,t}\vcentcolon=\delta\bar Z_{s,t}-\bar Z'_s\delta\bar X_{s,t}$.
Note that the right hand side of \cref{def.scrp.metricd} becomes infinite when $Z_0-\bar Z_0$ is not $L_m$-integrable.
If $X=\bar X$, we simply write $\| Z,Z';\bar Z,\bar Z'\|_{X;\beta,\beta';m}$. 

An additional concept which is introduced in \cite{FHL21} is that of \textit{stochastic controlled vector field}.
As is  seen from properties \ref{scvec.f}-\ref{scvec.remainder} below, these are themselves stochastic controlled processes taking values in the Banach space \( \mathcal Y=(\mathcal B_b(E;\bar W),|\cdot|_{\infty})\).
However, they come with additional regularity properties and so it is necessary to introduce new definitions.
The next one is a generalization of the concept introduced in \cite{FHL21}, as is seen by particularizing to the case \( E=\abx= W\).
\begin{definition}[Stochastic controlled vector fields]\label{def.scvec}
	Let \( E \) be a topological vector space and fix a continuously embedded Banach space \( \abx\subset E .\)
	Fix real numbers $\gamma>1$; $0\le\beta,\beta'\le1$; $m\in[2,\infty)$ and $n\in[m,\infty]$.
	\textit{We call $(f,f')$ a stochastic 
		controlled vector field on \( E \) along \(\abx \)} and write \( (f,f')\in\mathbf D^{\beta,\beta'}_XL_{m,n}\C^\gamma_{b,\abx}(E)\)
	if the following conditions are satisfied.
	\begin{enumerate}[(i)]
		\item\label{scvec.f}
		The pair
		$$
		(f,f'): \Omega \times [0,T]  \to \C^\gamma_{b,\abx}(E, \bar W) \times \C^{\gamma-1}_{b,\abx} (E,\LL(V,\bar W))
		$$
		is  progressively measurable in the strong sense and uniformly $n$-integrable, i.e.
		$$
		\sup_{s\in[0,T]}\||f_s|_{\gamma;\abx\subset E}\|_n +  \sup_{s\in[0,T]}\||f'_s|_{\gamma-1;\abx\subset E}\|_n < \infty\,.
		$$
		\item\label{scvec.brakets} 
		Putting
		\begin{align*}
			\bk{Z}_{\kappa;m,n}=\sup_{(s,t)\in \Delta(0,T)}\frac{\big\|\big\|\sup_{x\in E}|Z_{s,t}(x)|\,\big|\,\cff_s \big\|_m\big\|_n}{(t-s)^\kappa},
		\end{align*}
		the quantities
		$\bk{\delta f}_{\beta;m,n}$, $\bk{\delta f'}_{\beta';m,n}$, \( \bk{\delta Df}_{\beta';m,n} \) are finite. 
		
		This is of course equivalent to saying that \( \|Z\|_{\kappa;m,n} <\infty\) for each \( (Z,\kappa)\in \{(\delta f,\beta),(\delta f',\beta'),(\delta Df,\beta')\} \), where we think of \( Z \) as a stochastic process with values in \( (\mathcal B_b(E,\mathcal Y),|\cdot|_{\infty})\) for \( \mathcal Y\in \{\bar W,\LL(V,\bar W),\LL(\abx,\bar W)\} \).%
		\item\label{scvec.remainder} Putting  $R^f_{s,t}:=f_t-f_s-f'_s \delta X_{s,t}$ for each \( (s,t)\in \Delta \), then we have 
		\[
		\bk{\E_\bigcdot R^f}_{\beta+\beta';n}=\bk{\E_\bigcdot R^f}_{\beta+\beta';m,n}<\infty\,.
		\]
	\end{enumerate}
	For stochastic controlled vector fields \( (f,f')\in\mathbf D^{\beta,\beta'}_XL_{m,n}\C^\gamma_{b,\abx}(E)\), we introduce the semi-norms
	\begin{equation}
		\label{def.norms_scvf}
		\begin{aligned}
			&\|(f,f')\|_{\gamma;n}\vcentcolon=\sup_{s\in[0,T]}\||f_s|_{\gamma;\abx\subset E}\|_n+\sup_{s\in [0,T]}\||f'_s|_{\gamma-1;\abx\subset E}\|_n\,,
			\\&
			\bk{(f,f')}_{X;\beta,\beta';m,n}\vcentcolon=\bk{\delta f}_{\beta;m,n}+\bk{\delta Df}_{\beta';m,n}+\bk{\delta f'}_{\beta';m,n}+\bk{\E_\bigcdot R^f}_{\beta+\beta';m,n}.
		\end{aligned}
	\end{equation}
	We note that the dependence on the subspace $\abx$ is omitted in the above notation. We also abbreviate $\bk{(f,f')}_{X;\beta,\beta';m}$ for $\bk{(f,f')}_{X;\beta,\beta';m,m}$.

	Similarly, if $(f,f'),(\bar f,\bar f')$ are two stochastic controlled vector fields in $\mathbf{D}^{\beta,\beta'}_XL_{m,\infty}\C^\gamma_b$ and $\mathbf{D}^{\beta,\beta'}_{\bar X}L_{m,\infty}\C^\gamma_b$ respectively, we define 
	\begin{equation}
		\label{def.bk_metric}
		\begin{aligned}
			\bk{f,f';\bar f,\bar f'}_{X,\bar X;\beta, \beta';m}&=\bk{\delta f-\delta \bar f}_{\beta;m}+\bk{\delta f'-\delta \bar f'}_{\beta';m}+\bk{\delta Df-\delta D\bar f}_{\beta';m}
			\\&\quad
			+\bk{\E_\bigcdot R^f-\E_\bigcdot \bar R^{\bar f}}_{\beta+\beta';m}
		\end{aligned}
	\end{equation}
	and let $\bk{f,f';\bar f,\bar f'}_{X;\beta, \beta';m}=\bk{f,f';\bar f,\bar f'}_{X, X;\beta, \beta';m}$ accordingly.
\end{definition}


\subsection{McKean--Vlasov equations with rough common noise} 
\label{sec:main_results}

Before we proceed with further definitions and main results, a few comments on our probabilistic setting are in order.
Let $\W'=(\Omega',\mathcal G',\P',\{\mathcal F_t'\})$ be 
another stochastic basis, which is Polish and atomless. 
We assume that there is a measure-preserving map
\[
\begin{aligned}
	\pi\colon (\Omega',\mathcal G',\P')
	\longrightarrow 
	(\Omega,\mathcal G,\P).
\end{aligned}
\]
Consequently, if \( Z \) is a random variable on \( \Omega \), we can define
\begin{equation*}
	\slashed Z(\omega'):= Z(\pi(\omega')),\quad \quad \omega'\in\Omega',
\end{equation*}
and similarly for any events and algebra fields. Because of the measure-preserving property, the distribution of \( \slashed Z\) coincides with that of \( Z\). 
We assume that \( \mathcal G'\subset \slashed{\mathcal G} \) and in addition that \( \mathcal F_t' \supset\slashed{\mathcal F_t}\) for all  \(t\in [0,T]\) (see \cref{sec:measure_theoretic} for further details).

For notational convenience, we will put an apostrophe on other symbols associated to \( \W' \): namely $ L_q'(W)= L_q(\W',W)$, while $\E'$ is the expectation with respect to $\P'$ and \( \|\cdot\|'_q\) is the corresponding \( q \)-th metric as in \eqref{lebesgue_q}.


The following definition is due to Lions in  
case $q=2$ (see e.g. \cite{MR3752669} and references therein) allowing for other integrability exponents will be important to us and forces us to revisit a number of relevant properties that now require a different proof. (See the discussion before and around \eqref{law_inv_appendix}.)

\begin{definition}[Lions lift and slashed notation]\label{def.Lionslift} Recall that \( q\ge0. \)
	For a function \( f\in\mathcal B_b(\cpp_q(W),\aby) \), the Lions lift of $f$ with respect to \( (\Omega',\mathcal G',\P') \) is the function \( \hat f\in \mathcal B_b(L_q(\mathcal G';W),\aby) \) defined by 
		\[
		\hat f(\xi)=f(\law_{\Omega'}\xi),\quad \forall \xi \in L_q(\mathcal G';W).
		\] 
		We adopt a similar notation when \( f \) depends on additional parameters (for instance when \( f \) is a \( \mathcal B_b(\mathcal P_q(W),\aby) \)-valued random variable).
\end{definition}

In the remainder of the section/manuscript we consider measure-dependent coefficients, which are given under the form of Banach space-valued \( \{\mathcal F_t\} \)-progressive measurable stochastic processes (in the strong sense) 
\begin{equation}
	\label{measure-dependent_drift}
	\begin{aligned}
		&b\colon\Omega\times[0,T]\longrightarrow \mathcal B_b(W\times \cpp_q(W),W)\,,
		\\
		&\sigma\colon \Omega\times[0,T]\longrightarrow \mathcal B_b(W\times \cpp_q(W),\LL(\bar V, W))\,,
	\end{aligned}
\end{equation}
and similarly
\begin{equation}
	\label{measure-dependent_rough}
	\begin{aligned}
		&f\colon\Omega\times[0,T]\longrightarrow \mathcal B_b(W\times \cpp_q(W), \LL(V,W))\,,
		\\
		&f'\colon \Omega\times[0,T]\longrightarrow \mathcal B_b(W\times \cpp_q(W),\LL(V\otimes V, W))\,.
	\end{aligned}
\end{equation}
Their Lions lifts are denoted by \(\hat{b}, \hat{\sigma}, \hat{f}, \hat{f}'\) respectively.
We fix an additional real number 
\begin{equation*}
	p \in [1\vee q, m]\,.
\end{equation*}
Suppose that $(y,\xi)\mapsto\hat f_t(\omega,y,\xi)$ is continuously $W\times L'_p(W)$-Fréchet differentiable for every fixed pair \( (\omega,t) \in \Omega\times [0,T]\).
In particular, for each $(\omega,t,y,\xi)\in\Omega\times [0,T]\times  W\times  L'_q(W)$, the Fr\'echet derivative 
\[
\begin{aligned}
	&W\times L_p'(W)\longrightarrow \bar W
	\\
	&(z,\eta) \longmapsto D\hat f_t(\omega,y,\xi)[z,\eta]=D_1\hat f_t(\omega,y,\xi)[z] + D_2\hat f_t(\omega,y,\xi)[\eta]
\end{aligned}
\]
defines a bounded linear operator, where we henceforth write 
\( D_1 \) for the usual Fréchet derivative with respect to \( y\in W \), while \( D_2 \) indicates differentiation with respect to \( \xi \in  L'_q(W)\)  along \( \abx =L'_p(W)\) (see \cref{sec:along}). Note that $D_1\hat f_t(\omega,y,\xi)$ is an element of $\LL(W,\LL(V,W))$ while $D_2\hat f_t(\omega,y, \xi)$ is an element of $\LL(L'_p(W),\LL(V,W))$.

For random variables \(Y,Z\) and each \(\omega\in \Omega'\), the substitution 
\begin{align*}
    (y,\xi,z,\eta)\leftarrow (Y(\omega),\slashed{Y}, Z(\omega),\slashed{Z})
\end{align*} will appear frequently and hence, the notation
\begin{equation}
	\label{D_hat_f_f}
	\begin{aligned}
	\hat f_t(Y,\slashed{Y})(\omega) &:= \hat f_t(\omega,Y(\omega),\slashed{Y})\,,\\
	D\hat f_t(Y,\slashed{Y})[Z](\omega) &		
:=  D_1\hat f_t(\omega,Y(\omega),\slashed{Y})[Z(\omega)]+D_2\hat f_t(\omega,Y(\omega),\slashed{Y})[\slashed{Z}]
	\end{aligned}
\end{equation}
will be used as a shorthand.

\begin{example}
	For the simple case $\hat f(y, \xi)=\E' h(y,\xi)$, where $h$ is a $\C^3_b$-function on $W\times W$, we have $D_1\hat f(y,\xi)=\E' D_1h(y,\xi)$ and $D_2\hat f(y, \xi)[\eta]=\E' [D_2h(y,\xi)\eta]$. The expression $D\hat f(Y,\slashed{Y})[Z]$ with \( Z=\hat f(Y,\slashed{Y}) \) refers to the random variable
	\begin{align*}
		D\hat f(Y,\slashed{Y})[\hat f(Y,\slashed{Y})](\omega)
		&=
		D_1\hat f(Y(\omega),\slashed{Y})[\hat f(Y(\omega),\slashed{Y})]
		+D_2\hat f(Y(\omega),\slashed {Y})[\hat f(Y(\cdot),\slashed {Y})]
		\\&
		=\int_{\Omega'}\int_{\Omega'} D_1h(Y(\omega),\slashed Y(\omega'_1))h(Y(\omega),\slashed Y(\omega'_2))\P'(d \omega'_1)\P'(d \omega'_2)
		\\&\quad 
		+\int_{\Omega'}\int_{\Omega'} D_2h(Y(\omega),\slashed Y(\omega'_1))h(\slashed Y(\omega'_1),\slashed Y(\omega'_2))\P'(d \omega'_1)\P'(d \omega'_2).
	\end{align*}
\end{example}
We now give the definition of a solution to \eqref{eqn.MV}, driven by some rough path $(X,\XX)$.
\begin{definition}[Integrable solutions]
	\label{def.MVsoln}
	Let $m\in[2,\infty)$.
	A $L_{m,\infty}$-integrable solution of \eqref{eqn.MV}  over $[0,T]$ is an $\{\cff_t\}$-adapted continuous process $Y$ such that the following conditions are satisfied: 
	\begin{enumerate}[(i)]
		\item $\int_0^T |\hat b_r(Y_r,\slashed {Y_r})|dr$ and $\int_0^T |(\hat \sigma_r \hat \sigma^\dagger_r)(Y_r, \slashed {Y_r})|dr$ are finite a.s.;
		\item $(\hat f(Y,\slashed Y),D\hat f(Y,\slashed Y)[\hat f(Y,\slashed Y)]+\hat f'(Y,\slashed Y))$ is a stochastic controlled process in 
		$\DL{\bar\alpha,\bar\alpha'}{m}{\infty}([0,T],\W,\LL(V,W))$ 
		for some $\bar \alpha,\bar\alpha'\in[0,1]$, $\alpha+(\alpha\wedge\bar \alpha)>\frac12$, $\alpha+(\alpha\wedge\bar \alpha)+\bar \alpha'>1$; 
		\item $Y$ satisfies the following stochastic Davie-type expansion
		\begin{equation}\label{est.defJMV}
			\|\|J_{s,t}|\cff_s\|_m\|_\infty=o(t-s)^{1/2}
			\tand\|\E_sJ_{s,t}\|_\infty=o(t-s)
		\end{equation}
		for every $(s,t)\in \Delta$, where  
        \begin{equation}\label{def.JY}
            \begin{aligned}
                J_{s,t} &=\delta Y_{s,t}-\int_s^t\hat b_r(Y_r,\slashed {Y_r})dr-\int_s^t \hat \sigma_r(Y_r,\slashed {Y_r})dB_r
                \\&\quad-\hat f_s(Y_s,\slashed {Y_s})\delta X_{s,t}-\left(D\hat f_s(Y_s,\slashed {Y_s})[\hat f_s(Y_s,\slashed {Y_s})]+\hat f'_s(Y_s,\slashed {Y_s})\right)\bxx_{s,t}.
            \end{aligned}
        \end{equation}
	\end{enumerate}
	When the starting position $Y_0=\xi$ is specified, we say that $Y$ is a solution starting from $\xi$ (at time \(t=0\)).
\end{definition}

The regularity conditions on the coefficients of \eqref{eqn.MV} are stated together in the following assumption. Recall that \( m,p,q\) are given non-negative real numbers such that \(m\ge p\ge 1\vee q\).
\begin{assumption}[\( m,p,q \)]
	\label{assume.regularity} 
	\phantom{For spacing}
	
	\noindent	Let  $\gamma \in (\frac1 \alpha,3]$ and suppose that $m\ge 2$.
	Let \( b,\sigma,f,f' \) be progressive measurable, measure-dependent coefficients as in \eqref{measure-dependent_drift}-\eqref{measure-dependent_rough}.
	Suppose that 
	\begin{enumerate}[(a),series=assumption]
		\item\label{item.reg.bs} 
		For each $t\in[0,T]$ and \( \P \)-a.s.\
		the functions $ b_t(\omega,\cdot), \sigma_t(\omega,\cdot)$ belong to the class $\C^1_{b,p,q}=\C^1_{b,W\times L_p'}(W\times  L'_q)$ and satisfy 
		\begin{align}\label{con.gammamholder}
			\sup_{t\in [0,T]}\||\hat b_t|_{\C^1_{b,p,q}} \|_\infty
			+\sup_{t\in [0,T]}\||\hat \sigma_t|_{\C^1_{b,p,q}} \|_\infty<\infty.
		\end{align}
		
		\item The Lions lift $(\hat f,\hat f')$ forms a stochastic controlled vector field on \( E=W\times L_q'\) along \( \mathcal X=W\times L_p' \).
		Moreover we require that
		\[
		(\hat f,\hat f')\in \mathbf{D}^{ \alpha,\alpha}_XL_{m,\infty}\C^\gamma_{b,p,q} 
		\]
		while
		\[
		(D\hat f,D\hat f') \in\mathbf{D}^{ \alpha,\alpha}_XL_{m,\infty}\C^{\gamma-1}_{b,p,q}\,.
		\]
	\end{enumerate}
\end{assumption}
\begin{remark}
	Because of the embedding \(L'_m\hookrightarrow L'_p\), \cref{assume.regularity}-($m,p,q$) implies \cref{assume.regularity}-($m,m,q$)  (see also \eqref{id.cmq} below).
\end{remark}

We have now all in hand to state our first result.

\begin{theorem}\label{thm.MV.wellposed}
	Grant \cref{assume.regularity}-(\( m,m,q \)).
	Then, for every $\xi\in  L_q(\cff_0)$ and $T>0$, equation \eqref{eqn.MV} has a unique $L_{m,\infty}$-integrable solution\footnote{in the sense of Definition \ref{def.MVsoln}}
	 \( Y\colon \Omega\times [0,T]\to W\) starting from $Y_0=\xi$. 
	This solution forms a stochastic controlled rough path in \(\mathbf D_X^{\alpha,\alpha}L_{m,\infty}\) with \( Y'=\hat f(Y,\slashed Y) \).
\end{theorem}
The next result is a stability estimate which is the measure-dependent analogue of the stability result from \cite{FHL21}. We note that the hypotheses are slightly weakened in comparison with \cref{thm.MV.wellposed}. In particular we stress that \cref{assume.regularity} is not necessary.
\begin{theorem}[Stability]\label{thm.MKV.Stability}
	Fix real numbers \(\gamma\in(\frac{1}{\alpha},3]\); \(m\ge2\); $\beta,\beta'\in(0,\alpha)$ subject to 
	\[
	\alpha+(\gamma-1)\beta>1,
	\quad 1-\alpha-\beta <\beta'\le1
	\]
	and recall that \(q\in[0,\infty)\), \(m\ge p\ge q\vee1\).	
	Let $\xi,\bar \xi$ be in $ L_q(\cff_0)$; 
	$\X, \bar \X$ be in $\CC^\alpha$; \( \sigma,\bar\sigma,b,\bar b ,(f,f'),(\bar f,\bar f')\) be measure-dependent coefficients as in \eqref{measure-dependent_drift}-\eqref{measure-dependent_rough} such that 
	\begin{enumerate}[(a)]
		\item  $\xi-\bar \xi\in L_p$,
		\item \(b,\sigma\) satisfy part \ref{item.reg.bs} of  \cref{assume.regularity}-$(m,p,q)$,
		\item \((\hat f,\hat f')\in \mathbf D_{X}^{\beta,\beta}L_{m,\infty}\mathcal C^{\gamma}_{b,p,q}\) and \((D\hat f,D\hat f')\in \mathbf D_{X}^{\beta,\beta'}L_{m,\infty}\mathcal C^{\gamma-1}_{b,p,q}\),
		\item \(\sup_{t\in [0,T]}\||\hat b_t|_{\infty} \|_\infty
			+\sup_{t\in [0,T]}\||\hat \sigma_t|_{\infty} \|_\infty\) is finite,
		\item \((\hat {\bar f},\hat {\bar f}')\in\mathbf{D}^{\beta,\beta'}_{\bar X}L_{m,\infty}\C^{\gamma-1}_{b,p,q}\).
	\end{enumerate}
	Let $Y$ be a $L_{m,\infty}$-integrable solution to \eqref{eqn.MV} starting from $\xi$, and similarly denote by
	$\bar Y$ a $L_{m,\infty}$-integrable solution to \eqref{eqn.MV} starting from $\bar \xi$ with associated coefficients $(\bar b,\bar \sigma,\bar f,\bar f',\bar \X)$. Suppose in addition that $M<\infty$ exists such that 
	\begin{multline}
		\label{M_stability}
		|\X|_\alpha+|\bar\X|_\alpha+\sup_{t\in[0,T]}(\||\hat b_t|_{\mathcal C^{1}_{b,p,q}}\|_\infty+\||\hat \sigma_t|_{\mathcal C^1_{b,p,q}}\|_\infty+\||\hat f_t|_{\mathcal C^{\gamma}_{b,p,q}}\|_\infty+\||\hat f'_t|_{\mathcal C^{\gamma-1}_{b,p,q}}\|_\infty)
		\\
		+\bk{(\hat f,\hat f')}_{X; \beta,\beta;m,\infty}+\bk{(D\hat f,D\hat f')}_{X;\beta,\beta';m,\infty}\le M.
	\end{multline}
	
	Then, denoting by \( R^Y_{s,t}=\delta Y_{s,t}-f(Y_s)\delta X_{s,t}\) and \(\bar R^{\bar Y}_{s,t}=\delta\bar Y_{s,t}-\bar f(\bar Y_s)\delta\bar X_{s,t} \),
	we have the estimate 
	\begin{multline}
		\label{est.mkv.stability}
			\|\sup_{t\in[0,T]}|\delta Y_{0,t}-\delta \bar Y_{0,t}|\|_m+\| Y- \bar Y\|_{\alpha;m}+\|\delta f^Y(Y)-\delta\bar f^{\bar Y}(\bar Y)\|_{\beta;m} + \|\E_{\bigcdot} R^Y-\E_{\bigcdot}\bar R^{\bar Y}\|_{\alpha+ \beta;m}
			\\
			\lesssim\|\xi-\bar \xi\|_p+\rho_{\alpha,\beta}(\X,\bar \X)
			+\sup_{t\in[0,T]}\||\hat \sigma_t-\hat{\bar \sigma}_t|_\infty\|_m
			+\sup_{t\in[0,T]}\||\hat b_t-\hat{\bar b}_t|_\infty\|_m
			\\+\|(\hat f-\hat{\bar f},\hat f'-\hat{\bar f}')\|_{\gamma-1;m}+\bk{\hat f,\hat f';\hat{\bar f},\hat {\bar f}'}_{X,\bar X;\beta,\beta';m}\,,
	\end{multline}
	where the implied constant depends only upon \(\alpha,\beta,\beta',\gamma,T,M\) and where we recall notations \eqref{def.norms_scvf}, \eqref{def.bk_metric} and \eqref{def.rho_metric}.
\end{theorem}
We discuss some connections between our results with previous works. 
\begin{example}[Relation with \cite{MR4073682}]\label{ex.reg2}
	We compare \cref{assume.regularity} with the one in \cite{MR4073682}.
	Let $f:W\times\cpp_2(W)\to\LL(V,W)$ satisfy the Regularity Assumptions 1 and 2 of \cite{MR4073682} and take \( \gamma=3 \).
	Then $(f,0)$ satisfies \cref{assume.regularity}-(\( m,p,q \)) with \( q=2 \) and any real numbers \( m\ge p\ge4 \).
	Indeed, let $\partial_\mu$ denote the Lions derivative (see \cite[Chapter 5]{MR3752669}). Then there are measurable functions
	\begin{align*}
		&W\times\cpp_2(W)\times W\ni (x,\nu,v)\mapsto \partial_\mu f(x,\nu)(v)
		\\\shortintertext{and}
		&W\times\cpp_2(W)\times W\times W\ni(x,\nu,v,\bar v)\mapsto \partial^2_\mu f(x,\nu)(v,\bar v)
	\end{align*}
	such that		
	\begin{align*}
		D_2\hat f(x,\xi)[\eta]=\E'\left[ \partial_\mu f(x,\law(\xi))(\xi)\cdot \eta\right]
		, \quad
		D_1D_2\hat f(x,\xi)[\eta]=\E'\left[\partial_x \partial_\mu f(x,\law(\xi))(\xi)\cdot \eta\right]
	\end{align*}
	and
	\begin{multline}\label{id.D2f}
		D^2_2\hat f(x,\xi)[\eta_1,\eta_2]=\int_{\Omega'}\partial_v \partial_\mu f(x,\law(\xi))(\xi(\omega'))\cdot \eta_2(\omega')\otimes \eta_1(\omega')\P'(d \omega')
		\\+\int_{\Omega'}\int_{\Omega'}\partial^2_\mu f(x,\law(\xi))(\xi(\omega'_2), \xi(\omega'_1))\cdot \eta_2(\omega'_2)\otimes  \eta_1(\omega'_1)\P'(d \omega'_1)\P'(d \omega'_2)
	\end{multline}
	for every $\xi, \eta \in L'_2$ and $\eta_1, \eta_2 \in L'_2$, as follows from
{\cite[Proposition 5.85]{MR3752669}}. Thanks to Regularity Assumption 2 of
{\cite{MR4073682}} we see that the maps
\[ (x, \xi) \mapsto \partial^2_x f (x, \law (\xi)), \partial_x \partial_{\mu}
   f (x, \law (\xi)) (\xi), \partial_v \partial_{\mu} f (x, \law (\xi)) (\xi),
   \partial^2_{\mu} f (x, \law (\xi)) (\xi, \tilde{\xi}) \]
are bounded and Lipschitz from $W \times L'_2$ to $L'_2$ (here $\tilde{\xi}$
is an independent copy of $\xi$). Using these regularity conditions,
H{\"o}lder inequality and the above identities, it is straightforward to show
(cf. Example \ref{ex.lip}) that $\hat{f}$ belongs to $\C^3_{b, W \times L_p'} (W \times
L'_2)$ for any $p \ge 4$, verifying \cref{assume.regularity}-(b).
	%
	%
	%
\end{example}

\textbf{Connection with RSDEs.} 
For each stochastic controlled rough path $(\eta,\eta')$ defined on the basis $\W$  with $\eta_t(\omega)\in W$ and $\eta'_t(\omega)\in\LL(V,W)$, we define
\footnote{Note that for each $(\omega,t,y)$, $(f^\eta_t)'(\omega,y)$ is an element in $\LL(V,\LL(V,W))$ through the relation 
	\((f^\eta_t)'(\omega,y)[x]=D_2\hat f_t(\omega , y,\slashed\eta_t)[\slashed\eta'_tx]+\hat f_t(\omega,y,\slashed\eta_t)[x]\in\LL(V,W)\) for every \( x\in V. \)}
\begin{equation}
	\label{def.feta}
	\begin{aligned}
		&f^\eta_t(\omega,y)=\hat f_t(\omega,y,\slashed\eta_t),
		\\
		&(f^\eta_t)'(\omega,y)=D_2\hat f_t(\omega,y,\slashed{\eta}_t)[\slashed\eta'_t]+\hat f'_t(\omega,y,\slashed\eta_t)\,.
	\end{aligned}
\end{equation}
We now choose $(\eta,\eta')=(Y,\hat f(Y,\slashed Y))$. In such case, it follows that
\begin{align*}
	Df^\eta_s(Y_s)f^\eta_s(Y_s)+ (f^\eta_s)'(Y_s)
	=D\hat f_s(Y_s,\slashed {Y_s})[\hat f_s(Y_s,\slashed {Y_s})]+\hat f'_s(Y_s,\slashed {Y_s})\,,
\end{align*}
where the first term in the right hand side is understood in the sense of \eqref{D_hat_f_f}.
Therefore, the relation \eqref{def.JY} becomes
\begin{equation}\label{def.JYeta}
    \begin{aligned}
        \delta Y_{s,t}
        &=\int_s^t b^\eta_r(Y_r)dr+\int_s^t \sigma^\eta_r(Y_r)dB_r
    \\&\quad+f^\eta_s(Y_s)\delta X_{s,t}+\left(Df^\eta_s(Y_s)f^\eta_s(Y_s)+ (f^\eta_s)'(Y_s)\right)\XX_{s,t}
    +J_{s,t}\,.
    \end{aligned}
\end{equation}
Properties (i-iii) of \cref{def.MVsoln} can be rewritten as 
\begin{enumerate}[(i)]
    \item $\int_0^T | b^\eta_r(Y_r)|dr$ and $\int_0^T |( \sigma^\eta_r  (\sigma^\eta)^\dagger_r)(Y_r)|dr$ are finite a.s.;
    \item $( f^\eta(Y),D f^\eta(Y)[ f^\eta(Y)]+ (f^\eta)'(Y))$ belongs to  
    $\DL{\bar\alpha,\bar\alpha'}{m}{\infty}([0,T],\W,\LL(V,W))$ 
    for some $\bar \alpha,\bar\alpha'\in[0,1]$, $\alpha+(\alpha\wedge\bar \alpha)>\frac12$, $\alpha+(\alpha\wedge\bar \alpha)+\bar \alpha'>1$; 
    \item $Y$ satisfies the stochastic Davie-type expansion \eqref{est.defJMV}
    for every $(s,t)\in \Delta$, where \(J\) is defined in \eqref{def.JY}, equivalently in \eqref{def.JYeta}.
\end{enumerate}
These properties precisely mean that \(Y\) is a \(L_{m,\infty}\)-integrable solution to a rough stochastic differential equation (cf. \cite[Definition 4.2]{FHL21}) with vector fields \((b^\eta,\sigma^\eta,f^\eta,(f^\eta)')\) where $(\eta,\eta')=(Y,\hat f(Y,\slashed Y))$, more precisely, $Y$ satisfies the following rough stochastic differential system
\begin{equation}\label{eqn.MVrsde}
	\left\{
	\begin{aligned}
		&dY_t=b^\eta_t(Y_t)dt+\sigma^\eta_t(Y_t)dB_t+(f_t^\eta,(f^\eta_t)')(Y_t)d\X_t\,,
		\\&  (\eta,\eta')=(Y,\hat f(Y,\slashed Y))\,.
	\end{aligned}
	\right.
\end{equation}
The converse  argument is evident. We thus have shown the following
\begin{proposition}
    A process $Y$ is a \(L_{m,\infty}\)- integrable solution to the rough McKean--Vlasov equation \eqref{eqn.MV} if and only if it is  a \(L_{m,\infty}\)-integrable solution to the rough stochastic differential system \eqref{eqn.MVrsde}.
\end{proposition}  

\begin{remark}\label{rmk.2ndway}
	The previous discussion indicates that \eqref{eqn.MV} is equivalent to the rough stochastic differential equation
	\begin{align*}
		dY_t=\hat b_t(Y_t,\slashed Y_t)dt+\hat \sigma_t(Y_t,\slashed Y_t)dB_t+(\hat f_t,\hat f'_t)(Y_t,\slashed Y_t)d\X_t,
	\end{align*}
    with vector fields \(\hat b,\hat \sigma,\hat f,\hat f'\) which are defined on \(\Omega\times\R_+\times W\times L'_q(W)\).
	From this perspective, the methods from \cite{FHL21} could be employed in more direct ways to treat \eqref{eqn.MV}. This approach avoids fixing the law components in the coefficients as in \eqref{eqn.MVrsde}. Nevertheless, well-posedness for McKean--Vlasov equations is classically obtained through fixing the law components (see \cite{MR1108185}) and we have followed this classical approach in the current article.
\end{remark}

\subsection{From rough to stochastic common noise} \label{sec:FR2SCN}

In the literature of McKean--Vlasov equations with common noise 
the latter is typically modeled by some
Brownian motion $W$, independent of the idiosyncratic noise that affects
individual particles, modeled by i.i.d. realizations of some Brownian noise
$B$. In the mean-field limit generic dynamics are then of the form
\begin{equation}
  \label{eqn.MV_brownian} \left\{ \begin{aligned}
    & dY_t (\omega) = b_t (Y_t (\omega), \mu_t) \hspace{0.17em} dt + \sigma_t
    (Y_t (\omega), \mu_t) \hspace{0.17em} dB_t (\omega) + f_t (Y_t (\omega),
    \mu_t) \hspace{0.17em} dW_t (\omega),\\
    & \mu_t = \mathrm{Law} (Y_t \mid W), \quad Y_0 (\omega) = \xi (\omega) .
  \end{aligned} \right.
\end{equation}

To relate this to the randomization of McKean--Vlasov SDEs with rough common
noise $\mathbf{X}$, suitably randomized, independent of $B$. To this end, let
$\tmmathbf{\Omega}' = (\Omega', \mathfrak{G}', (\mathfrak{F}'_t)_t,
\mathbb{P}')$ be a complete probability space,\footnote{In slightly abusive notation, $\tmmathbf{\Omega}'$ is different from the auxiliary space employed in Section \ref{sec:main_results}.} 
 and similar for
$\tmmathbf{\Omega}''$. Write
\[ \tmmathbf{\Omega}= (\Omega, \mathfrak{F}, (\mathfrak{F}_t)_t, \mathbb{P})
\]
for the (completed) product space, where $\Omega = \Omega' \times \Omega''$,
$\mathbb{P}=\mathbb{P}' \otimes \mathbb{P}''$ and $\mathfrak{F}_t =
(\mathfrak{F}'_t \otimes \mathfrak{F}''_t) \vee \mathfrak{N}$ where
$\mathfrak{N}$ denotes the $\mathbb{P}$-negligible sets on $\Omega$. Write
$\omega = (\omega', \omega'')$ accordingly. Consider now the unique solution
$Y^{\X} (\omega')$, constructed on the 
stochastic basis
$\tmmathbf{\Omega}'$, from  Theorem~\ref{thm.MV.wellposed}.
\begin{equation}
  \label{eqn.MVX} \left\{ \begin{aligned}
    & dY^{\X}_t (\omega') = b_t (Y^{\X}_t (\omega'),
    \mu^{\X}_t ; \X) \hspace{0.17em} dt + \sigma_t
    (Y^{\X}_t (\omega'), \mu^{\X}_t ; \X)
    \hspace{0.17em} dB_t (\omega')\\
    & \quad\quad\quad\quad \quad + (f_t, f'_t) (Y^{\X}_t (\omega'),
    \mu^{\X}_t ; \X) \hspace{0.17em} d \X_t,\\
    & \mu^{\X}_t = \mathrm{Law} (Y^{\X}_t), \quad
    Y^{\tmmathbf{X}}_0 (\omega') = \xi (\omega')
  \end{aligned} \right.
\end{equation}

{\assumption{\label{ass:4} Let $(\mathscr{C}_T,\rho)$ be a  Polish metric space with Borel $\sigma$-field $\mathfrak{C}_T$ such that  $(\mathscr{C}_T,
\rho) \hookrightarrow (\mathscr{C}_T^{\alpha}, \rho_{\alpha})$, $\alpha \in (1
/ 3, 1 / 2)$. Let 
\[ \bm{W} : (\Omega'', \mathfrak{F}_T'', \mathbb{P}'') \rightarrow
   (\mathscr{C}_T, \mathfrak{C}_T), \]
be adapted in the sense that $\bm{W}_t \in \mathfrak{F}''_t$. (In other
words, $\mathfrak{F}^{\bm{W}}_t \subset \mathfrak{F}_t''$.)}}

\medskip
Both $Y^\X$ and $\bm{W}$ can be ``lifted'' to $\Omega$, by setting $Y^{\X} (\omega)
\assign Y^{\X} (\omega')$ and $\bm{W} (\omega) \assign \bm{W}
(\omega'')$, respectively, where they are, by construction, independent under
$\mathbb{P}$. We also lift $\mathfrak{F}'_t$ (and similarly
$\mathfrak{F}''_t$) in the natural way, in the sense that
\[ A' \in \mathfrak{F}'_t \rightsquigarrow A' \times \Omega'' \in
   \mathfrak{F}'_t \otimes \mathfrak{F}''_0 . \]
We know from \cite{FLZ25} that under the appropriate conditions,
a rough It{\^o} process $\{Y^{\X} : \mathbf{Y} \in \mathscr{C}_T
\}$ over $\tmmathbf{\Omega}'$ admits a $\mathfrak{C}_T$-optional (resp.
$\mathfrak{C}_T$-progressive) version, as before coefficient fields $\{A,
\Sigma, (F, F')\}$ denote its coefficient fields.

\begin{proposition}
  Under the assumptions of  Theorem~\ref{thm.MV.wellposed}, with
  $\mathfrak{C}_T$-progressive coefficient fields, the family
  $\{Y^{\X} : \X \in \mathscr{C}_T \}$ has a
  $\mathfrak{C}_T$-progressive version.
\end{proposition}

\begin{proof}
  This was discussed in detailed in \cite{FLZ25} so we content ourselves with a
  sketch. By the very integral equation satisfied by $Y^{\X}$, it
  constitutes a rough It\^o process in the sense of \cite{FHL21, FLZ25}. It then
  suffices to check that all the coefficients fields $b, \sigma, (f, f')$ that
  appear in the Lebesgue, It\^o and rough stochastic integrals, respectively,
  have the correct $\mathfrak{C}_T$-progressive measurability. Since these
  involve $Y^{\X}$, we proceed by Picard approximation. Following the
  proof of Theorem~\ref{thm.MV.wellposed}
  one can do so by a two-step argument, as carried out in Section \ref{sec:proofs_MV_wp}.
That is, first
  solve we solve \ \eqref{eqn.MVrsde} for a fixed stochastic controlled rough
  path $(\eta^{\tmmathbf{X}}, (\eta^{\tmmathbf{X}})')$, which has the correct
  $\mathfrak{C}_T$-progressive measurability. The existence of a
  $\mathfrak{C}_T$-progressive version follows then from the results of
  [FHL25] for \ $\mathfrak{C}_T$-progressive versions of RSDEs. The solution
  to \eqref{eqn.MVX} is then constructed by another fixed point argument, the
  iteration of which again involves solving RSDEs, which, upon choosing a good
  version, preserves the desired $\mathfrak{C}_T$-progressive measurability.
  This iteration is contractive, hence convergent for $T \ll 1$, the limit
  hence also $\mathfrak{C}_T$-progressive. Finally, it is easy exercise that
  the (finite) patching required for a global solution on some given interval
  $[0, T]$ is also $\mathfrak{C}_T$-progressive.
  
  Rather than spelling out all details here, we mention two alternative
  proofs. As noted in Remark \ref{rmk.2ndway}, one can avoid the two-step argument in setting
  up a single Picard iteration using the Lions lift, the above construction
  then simplifies accordingly. Yet another argument, as noted in \cite{FLZ25}, is
  based on propagation of chaos, where one can rely directly on known results
  for RSDEs \cite{FLZ25} and then pass to the mean-field limit, which of course
  preserves measurability properties. \ 
\end{proof}

Unless otherwise mentioned we always work with the version of
$\{Y^{\X} : \X \in \mathscr{C}_T \}$ supplied by the previous
result. Together with the assumptions on $\tmmathbf{W} (\omega'')$, a
{\em measurable} process is then obtained by randomization
\[ \bar{Y} (\omega) \assign Y^{\X (\omega')} |_{\X =
   \bm{W} (\omega'')} . \]
The remaining details are then exactly as in \cite{FLZ25}. In particular, assuming
causal (or no direct) dependence on $\tmmathbf{X}$ and assuming that the
random rough path $\bm{W} (\omega'')$ is the It\^o lift of a Brownian motion
$W (\omega'')$ it follows that $\bar{Y} (\omega)$ is a solution to
\eqref{eqn.MV_brownian}, with, thanks to Theorem 3.6 in \cite{FLZ25}, 
\[ \mu_t = \mathrm{Law} (\bar{Y}_t \mid W) = \tmop{Law} (Y^{\tmmathbf{X}}_t)
   |_{\X = \bm{W} (\omega'')} . \]
In view of (known) uniqueness results for strong solutions of
\eqref{eqn.MV_brownian}, it is then clear that our construction, if
specialized as explained, covers exactly the classical situation of
McKean--Vlasov with common noise. 

\section{Lions-type lifts and law-invariance}
\label{sec.preparatory_material}
\label{sec:frechet_regularity}

In the whole section we fix three real numbers 
\begin{equation}
	m\ge p\ge1\vee q\,,
	\quad \quad 
	q\ge0\,.
\end{equation} 

Let \( \Omega=(\Omega, \mathcal G,\P) \) be a probability space.
Recall that \(  L_q=L_q(\mathcal G;W) \) denotes the space of \( W\)-valued random variables of order \( q \) on \(\Omega\), topologized by the distance \( (\xi,\eta)\mapsto\|\xi-\eta\|_q \).

Recalling the notations of \cref{exa:regularity_rv}, we will now discuss some further properties of the classes
\[
\C^\gamma_{b,p,q}=\C^\gamma_{b,L_p}( L_q,\bar W)\,.
\] 
First, we observe from \cref{lem:embedding} the continuous embeddings
\begin{align}
	&\C^\gamma_{b,p,q}\hookrightarrow \C^\gamma_{b,p,p}
	\hookrightarrow \C^\gamma_{b,m,p}.
	\label{id.cmq}
\end{align}

It is clear that the first embedding in \eqref{id.cmq} is strict whenever $p>q$. So is the right-most one, as indicated by the following example.
\begin{example}\label{ex.lip}
	Let $h\colon \R\to\R$ be a bounded smooth function with bounded derivatives and define the function $g\colon L_2\to\R$ by   $g(\xi)=\E [h(\xi)]$. While this appears to 
	be a nice function of the law of $\xi$, it is known from \cite[Remark 5.80]{MR3752669} that there exists a  compactly supported smooth function $h$ such that $g$ is not twice continuously differentiable in Fr\'echet sense, 
	hence does not belong to $\C^\gamma_{b,2,2}$, whenever $\gamma > 2$.\footnote{Attention for $\gamma = 2$: per our definitions, $\C^2_{b,2,2} = \C^2_{b,2}$ does not mean twice continuously differentiable but continuously differentiable with Lipschitz derivative.
	In the particular example, $D g (\xi) = h' (\xi)$ which is clearly Lipschitz in $\xi \in L_2$, since $h'$ is Lipschitz.}
	We thus want to understand when $g \in \mathcal{C}^{\gamma}_{b, p, 2}$ for $\gamma
\in (2, 3]$ and $p \ge 2$. From 
\[ D^2 g (\xi) : L_p \times L_p \ni (A, B) \mapsto \mathbb{E} (h'' (\xi) A B),
\]
H\"older inequality gives $| \mathbb{E} (h'' (\xi) A B) | \leqslant \| h'' (\xi) \|_r
\| A \|_p \| B \|_p $ with $1 / r + 2 / p = 1$. Since $h''$ is bounded, this
always applies, even with $r = + \infty$ and $p = 2$, and shows that $D^2 g
(\xi)$ is bilinear continuous and can indeed, by duality, by identified with
$h'' (\xi) \in L_r$. To obtain the modulus of $L_2\ni \xi \mapsto \| h''
(\xi) \|_r$ along $L_p$, we take $\xi,\tilde\xi\in L_2$ such that  $\xi - \tilde \xi \in L_p$ and consider two cases. When $\infty >  r \ge p$ (equivalently $2 < p\le 3$), we have (using boundedness and Lipschitzness of $h''$,
then geometric interpolation)
		\[ \mathbb{E} | h'' (\xi) - h'' (\tilde{\xi}) |^r \lesssim \mathbb{E} | \xi -
   \tilde{\xi} |^p \Rightarrow \| h'' (\xi) - h'' (\tilde{\xi}) \|_{r}
   \lesssim (\| \xi - \tilde \xi \|_{p} )^{\frac{p}{r}} \]
   so that $\gamma= 2 + p/r = 2 + p (1 - 2 / p) = p$.
   When $r< p$ (equivalently $p>3 $), we have
   \[ \| h'' (\xi) - h'' (\tilde{\xi})\|_r \lesssim \| \xi -
   \tilde{\xi} \|_p \]
   so that $\gamma= 2 + 1 =3$. 
   In both cases, we have $g\in \C^{p\wedge3}_{b,p,2}$. 

   In fact, by considering higher derivatives, it is straightforward to verify that $g\in \C^{p}_{b,p,2}$ for all $p > 2$. 
   Furthermore, for this example, the ambient  space $L_2$ plays a minor role,  we can replace it by any $L_q$ with $0\le q\le p$, and we have $g\in \C^{p}_{b,p,q}$ even for general $p \in [1,\infty)$. Indeed, the case $p \in (1,2]$ only involves $Dg (\xi) = h' (\xi)$ and is elementary to discuss, the case $p=1$ even easier. 
   
	
\end{example}


%
%
Importantly for our needs, the following integral representation for the increments of a map that is differentiable  along \( L_p\subset L_q. \)
If \( \aby \) is a real Banach space and $g\colon L_q\to\aby$ forms an element of \( \mathcal C^\gamma_{b,p,q} \) for some \(\gamma>1\), then it follows from the mean-value theorem (\cref{thm.mvt}) that
\begin{align}
	g(\xi+\eta)-g(\xi)=\int_0^1 Dg(\xi+ \theta\eta)[\eta]d\theta
	\label{id.mvt}
	\quad \quad 
	\text{ for all }\xi\in L_q \enskip\text{and} \enskip \eta\in L_p\,.
\end{align}

Although Lions lifts per \cref{def.Lionslift} depend on the choice of probability space $\W$, they are law-invariant, i.e. \(\hat f(\xi)=\hat f(\bar \xi)\) whenever \(\law(\xi)=\law(\bar \xi)\). The following result shows that this property extends to their derivatives and certain conditioning.

\begin{proposition}[Law-invariance]\label{lem.lawDg}
	Suppose that $f\colon \cpp_q(W)\to\aby$ is such that  
	\[
	g\colon L_q\to\aby,\quad 
	\xi\mapsto g(\xi):=f(\law_{\Omega}\xi)
	\] 
	is continuously $L_p$-Fr\'echet differentiable 
	(recall that we use the shorthand \( L_q=L_q(\mathcal G;W)\)).
	Let $(\xi,\eta)\in  L_q\times L_p$.
	\begin{enumerate}[(i)]
		\item\label{item.lawinvariant} 
		For every \( (\bar\xi,\bar\eta)\in  L_q\times L_p $ satisfying $\law(\xi,\eta)=\law(\bar\xi,\bar\eta)\),
		we have
		\[
		 Dg(\xi)[\eta]=Dg(\bar\xi)[\bar\eta] \,.
		\]
		
		\item\label{item.conditionalinvariant} 
		If \( (\Omega,\cgg,\P) \) is atomless and $\cff\subset\cgg$ is a $\sigma$-algebra such that $\xi$ is \( \cff \)-measurable, then
		\begin{equation}
            \label{law_inv_appendix}
            Dg(\xi)[\eta]=Dg(\xi)[\E(\eta|\cff)]\,.
        \end{equation}
\end{enumerate}
		Consequently, if \( (\Omega,\cgg,\P) \) is atomless and $g$ is twice continuously differentiable along \( L_p \), we also get:
	\begin{enumerate}[(i),resume]
		\item\label{item.lawD2} 
		For every \( (\xi,\eta_1,\eta_2)\in L_q(\mathcal F)\times L_p\times L_p(\mathcal F) \).
		\[
		 D^2g(\xi)[\eta_1,\eta_2]=D^2g(\xi)[\E(\eta_1|\cff),\eta_2] \,.
		\]
	\end{enumerate}
\end{proposition}
To illustrate how this result connects with the properties of L-derivatives, suppose for the sake of argument that \( p=q=2\),
(in which case  \( g\) is clearly L-differentiable in the sense of \cite[Def.~5.22]{carmona2018probabilistic}). Then, for any \( \xi,\eta\in L_2 \), we have the presentation
\[
\begin{aligned}
	Dg(\xi)[\eta]
	&=\E[\partial_\mu f(\law(\xi))(\xi)\cdot\eta]
	\\
\end{aligned}
\]
from which \eqref{law_inv_appendix} follows trivially by conditioning on \(\cff\).

\begin{proof}[Proof of \cref{lem.lawDg}]
	By the Hahn--Banach Theorem, we can assume without loss of generality that \( \mathcal Y=\R \). Clearly, part \ref{item.lawinvariant} is a consequence of 
	\[
	D g(\xi)[\eta]=\lim_{\theta\downarrow0}\frac{g(\xi+\theta\eta)- g(\xi)}{\theta}=\lim_{\theta\downarrow0}\frac{g(\bar\xi+\theta\bar\eta)-g(\bar\xi)}{\theta}
	=D g(\bar\xi)[\bar\eta].
	\] 
	Part \ref{item.lawD2} follows from \ref{item.conditionalinvariant} and the relation
    \begin{align*}
        D^2g(\xi)[\eta_1,\eta_2]=\lim_{\theta \to 0}\frac{Dg(\xi+\theta \eta_2)[\eta_1]-Dg(\xi)[\eta_1]}{\theta} . 
    \end{align*}
	Furthermore, because \(L_\infty\) is dense in \(L_p,L_q\) and \(g\) is continuously $L_p$-Fr\'echet differentiable, it suffices to show \eqref{law_inv_appendix} assuming additionally that  \( \xi,\eta\in L_\infty. \)
	
	When \( \xi \) is a deterministic constant, property \eqref{law_inv_appendix} is a consequence of \cite[Lemma 4.2]{MR2277714}. Indeed, in this case, the set $C=\{\bar \eta\in L_\infty(W):D g(\xi)[\bar \eta]=D g(\xi)[\eta]\}$ is a closed convex, law invariant subset of $L_\infty$, and hence $\E(\eta|\cff)\in C$ by the aforementioned result. 
	When $\xi$ is a genuine random variable, $C$ is no longer convex or law invariant, this argument breaks down and we proceed in the following way.  
	\smallskip

	\textit{Step 1.}
	Consider first the case $\cff=\sigma(G_1,\ldots,G_k)$ where $G_i$'s are some disjoint measurable events 
	such that $\P(G_i)>0$ and $\Omega=\cup_{i=1}^k G_i$. 
	Since $\xi\in\cff$, it has the form $\xi=\sum_{i=1}^k c_i\mathds{1}_{G_i}$ 
	where $c_i$'s are some finite constants. Let $\varepsilon>0$ be fixed. 
	For each $i=1,\ldots,k$, $(G_i,\cgg_i,\P_i)$ is an atomless probability space where $\cgg_i=\{A\cap G_i\vert A\in\cgg\}$ and $\P_i:\cgg_i\to[0,1]$, $A\mapsto \P(A)/\P(G_i)$. Reasoning as in \cite[Lemma 1.3, step 2]{MR2651517}, we obtain some random variable $\eta^i=\frac1{n!}\sum_{s\in S_n}\eta^{i,s}$ on $(G_i,\cgg_i,\P_i)$ such that $\law(\eta^{i,s})=\law(\eta\vert_{G_i})$ under $\P_i$ for every $s\in S_n$ and $\|\eta^i- \E_{\P_i}(\eta\vert_{G_i})\|_\infty\le \varepsilon$.
	(Here $n=n(\varepsilon)$ is some positive integer and $S_n$ is the group of permutations on $\{1,\ldots,n\}$.)
	Since $\xi\vert_{G_i}=c_i$ is a constant, we have for every $s\in S_n$,
	\begin{align}\label{tmp.equallaw}
		\law(\xi\vert_{G_i},\eta^{i,s})=\law(\xi\vert_{G_i},\eta\vert_{G_i}) \quad\text{under}\quad\P_i.
	\end{align}
	Define
	\begin{align*}
		\bar \eta\vcentcolon=\sum_{i=1}^k \eta^i\mathds{1}_{G_i}=\frac1{n!}\adjustlimits\sum_{s\in S_n}\sum_{i=1}^k \eta^{i,s}\mathds{1}_{G_i}
	\end{align*}
	and note that $\xi=\sum_{i=1}^k \xi\vert_{G_i}\mathds{1}_{G_i} $ trivially.
	From \eqref{tmp.equallaw}, we have for every $s\in S_n$,
	\begin{align*}
		\law\Big(\xi,\sum_{i=1}^k \eta^{i,s}\mathds{1}_{G_i}\Big)=\law(\xi,\eta) \quad\text{under}\quad \P.
	\end{align*}
	Applying part \ref{item.lawinvariant}, we have $D g(\xi)\left[\sum_{i=1}^k \eta^{i,s}\mathds{1}_{G_i}\right]=D g(\xi)[\eta]$ for every $s\in S_n$,
	and hence
	\begin{align}\label{tmp.Dgbareta}
		D g(\xi)[\bar \eta]= D g(\xi)\left[\frac1{n!}\adjustlimits\sum_{s\in S_n}\sum_{i=1}^k \eta^{i,s}\mathds{1}_{G_i}\right]=D g(\xi)[\eta].	
	\end{align}
	On the other hand, we have	
	\begin{align*}
		\|\bar\eta-\E(\eta|\cff)\|_\infty=\left \|\bar \eta-\sum_{i=1}^k\E_{\P_i}[\eta\vert_{G_i}]\mathds{1}_{G_i}\right \|_\infty\le \varepsilon.
	\end{align*}
	Hence, we can send $\varepsilon\to0$ in \eqref{tmp.Dgbareta} to obtain that $D g(\xi)[\eta]=D g(\xi)[\E(\eta|\cff)]$.\smallskip
	
	\textit{Step 2.}
	For a general sub-$\sigma$-algebra $\cff\subset\cgg$, we may find a sequence of finite sub-$\sigma$-algebras $\{\cff_k\}_{k\ge0}$ of $\cgg$ such that 
	\begin{align*}
		\lim_{k\to\infty}\left[\|\E(\xi|\cff)-\E(\xi|\cff_k)\|_\infty+\|\E(\eta|\cff)-\E(\eta|\cff_k)\|_\infty\right]=0.
	\end{align*}
	Since $\xi\in\cff$, we can replace $\E(\xi|\cff)$ by $\xi$ in the above limit. From the previous step, we have
	\begin{align*}
		D g(\E(\xi|\cff_k))[\eta]
		=D g(\E(\xi|\cff_k))[\E(\eta|\cff_k)].
	\end{align*}
	Sending $k\to\infty$, using the above limit and continuity of $D g$, we obtain that $D g(\xi)[\eta]=D g(\xi)[\E(\eta|\cff)]$, completing the proof of part \ref{item.conditionalinvariant}.	
\end{proof}


\section{Composition lemmas}
\label{sec.comp}

Suppose that \(\W=(\Omega, \mathcal G,\{\cff_t\},\P) \), \(\W'=(\Omega', \mathcal G',\{\cff_t\},\P') \) are two Polish, filtered probability spaces and introduce
\[
\W\otimes \W'
=(\Omega\times \Omega',\mathcal G\otimes \mathcal G',\{\mathcal F_t\otimes \mathcal F_t'\}_{t\ge0},\mathbb P\otimes \mathbb P')\,,
\]
whose expectation is denoted by \( \E\otimes \E' \).
We are going to introduce and analyze an extension of the vector fields $(f^\eta,(f^\eta)')$ introduced in \eqref{def.feta} for stochastic controlled processes $(\eta,\eta')$ on the product space \( \boldsymbol{\Omega}\otimes \boldsymbol{\Omega'}\).

Let $(\eta,\eta')$ be a stochastic controlled rough path in $\mathbf{D}^{\beta,\beta'}_XL_{p}([0,T],\W\otimes\W';W)$, and $(f,f')$ be a measure dependent stochastic process such that $(\hat f,\hat f')\in\mathbf{D}^{\beta,\beta'}_XL_{m,\infty}\C^\gamma_{b,p,q}$ for some $\beta,\beta'\in[0,1]$ and \(\gamma>1\).
Suppose that $ \eta$ is pathwise continuous and \( q \)-integrable. 
We work with a version of \( \eta \) (whose existence is ensured by Fubini theorem) such that \( \eta_t(\omega,\cdot) \) forms an element of \( L_q'\) for \( \P \)-almost every \( \omega\in \Omega \) and every \( t\in [0,T] \). 
In this setting, we introduce for each \( \omega\in \Omega\), \( t\in[0,T]\) and \( y\in W\)
\begin{equation}\label{def.geta}
	\left \{\begin{aligned}
		&f\comp\eta_t(\omega,y)=\hat f_t(\omega,y,\eta_t(\omega,\cdot))\,,
		\\
		&(f\comp\eta)'_t(\omega,y)=D_2\hat f_t(\omega,y,\eta_t(\omega,\cdot))[\eta'_t(\omega,\cdot)]+\hat f'_t(\omega,y,\eta_t(\omega,\cdot))
		\,,
	\end{aligned}
	\right .
\end{equation}
and we shall drop the dependence on \( \omega \) when there is no risk of confusion.
When $(\eta,\eta')$ is independent of $\omega'$, we note that the vector fields \( (f^\eta,(f^{\eta})') \) introduced in \eqref{def.feta} satisfy
\begin{align}\label{id.expiscomp}
	(f^\eta,(f^\eta)')=(f\comp \slashed\eta,(f\comp \slashed\eta)')
\end{align}
where we recall that \( \slash \) is the isometry introduced in \eqref{slashing}.
For the treatment of McKean--Vlasov equation with rough common noise, equation \eqref{eqn.MV}, one only needs to consider vector fields of the type \( (f^\eta,(f^{\eta})') \). However, in a subsequent work, we are going to establish convergence of particle systems toward \eqref{eqn.MV}, in which the generality of \eqref{def.geta} is necessary.

We now state an analogous result to \cite[Lemma 3.20]{FHL21}, which shows in particular that the vector fields in 
\eqref{def.geta} are stochastically controlled.
In comparison to \cite{FHL21}, because of the different domains of the vector fields, we need to work with a different set of integrability parameters, which are reflected in the following notation
\begin{align*}
	&\|\|\delta \eta\|_p'\|_{\beta;m,\infty}=\sup_{(s,t)\in \Delta}\frac{\|\|\|\delta \eta_{s,t}\|_p'|\cff_s\|_m\|_\infty }{(t-s)^\beta},
	\\& \|\|\eta'\|'_p\|_{\beta';m,\infty}=\sup_{t\in[0,T]}\|\|\eta'_t\|_p'\|_\infty+ \sup_{(s,t)\in \Delta}\frac{\|\|\|\delta \eta_{s,t}'\|_p'|\cff_s\|_m\|_\infty }{(t-s)^{\beta'}},
	\\& \|\|(\E\otimes \E')_\bigcdot R^\eta\|'_p\|_{\beta+\beta';\infty}=\sup_{(s,t)\in \Delta}\frac{\|\|\E_s\otimes\E'_sR^\eta_{s,t}\|_p'\|_\infty }{(t-s)^{\beta+\beta'}}.
\end{align*}
%
\begin{lemma}\label{prop.newcvec}
	Fix parameters $\beta,\beta'\in[0,1]$, \( \gamma>1\).
	Let $(\eta,\eta')$ be a stochastic controlled rough path in \(\mathbf{D}^{\beta,\beta'}_{X}L_{p}([0,T];\boldsymbol\Omega\otimes\boldsymbol{\Omega'};W)\) such that $\eta_t(\omega,\cdot)\in L'_q$ for every $t\in[0,T]$ and $\P$-a.s.\ $\omega$. Assume that there is a non-negative finite number $M$ such that 
	\begin{align}\label{con.eta1new}
		\|\|\delta\eta\|'_p\|_{\beta;m,\infty}+ \|\|\eta'\|'_p\|_{\beta';m,\infty}+\|\|(\E\otimes \E')_\bigcdot R^\eta\|'_p\|_{\beta+\beta';\infty}\le M.
	\end{align}
	Take \( n\in [m,\infty] \), \( m\ge2 \) and 
	fix measure-dependent coefficients  
	\begin{equation}
		\label{measure-dependent-lemma}
		(f,f')\colon\Omega\times[0,T]\longrightarrow \mathcal B_b\big(W\times \cpp_q(W), \bar W\big)
		\times \mathcal B_b\big(W\times \cpp_q(W),\LL(V, \bar W)\big)
	\end{equation}
	such that \( (\hat f,\hat f') \) forms a stochastic controlled vector field in $\mathbf{D}^{\beta,\beta'}_XL_{m,n}\C^\gamma_{b,p,q}$. Define  $f\comp\eta$ and $(f\comp\eta)'$ according to \eqref{def.geta}.
	Then there is a finite constant $C=C(\beta,\beta',\gamma,T)$ such that 
	\begin{align}
		\label{est:stability_spatial}
		\|(f\comp\eta,(f\comp\eta)')\|_{\gamma;n}
		&\le C\|(\hat f,\hat f')\|_{\gamma;n}(1+M),
		\\
		\label{est:stability_temporal}
		\bk{(f\comp\eta,(f\comp\eta)')}_{X;\beta,\beta'';m,n}
		&\le C(\|(\hat f,\hat f')\|_{\gamma;n}+\bk{(\hat f,\hat f')}_{X;\beta,\beta';m,n})(1+M)^{\gamma\wedge 2} \,,
	\end{align}
	where \(\beta''=[(\gamma-1)\beta]\wedge\beta'\wedge\beta\).
	Consequently, $\left(f\comp \eta,(f\comp \eta)'\right)$ is a stochastic controlled vector field in $\mathbf{D}^{\beta,\beta''}_XL_{m,\infty}\C^\gamma_b(W)$.
\end{lemma}
\begin{proof}
	The former estimate is straightforward, using \eqref{est.normDkg}, \eqref{est.normdgammag} and the fact that $\eta_t(\omega,\cdot)\in  L_q(\Omega';W)$ for each $t$ and \( \P \)-a.e.\ $\omega$. 
	Indeed, we have \(|(f\comp \eta)_t|_\gamma\le |\hat f_t|_{\gamma}\) and $|(f\comp \eta)'_t |_{\gamma-1}\le|D_2\hat f_t|_{\gamma-1}\|\eta'_t\|'_p+|\hat f'|_{\gamma-1}$ which yield  \eqref{est:stability_spatial}.\smallskip

	To show the latter estimate, we can assume without loss of generality that $\gamma\in(1,2]$ and $\|(\hat f,\hat f')\|_{\gamma;n}+\bk{(\hat f,\hat f')}_{X;\beta,\beta';m,n}\le1$.
	For notational simplicity, we write $\hat f_t(\omega,\eta_t)=[y\mapsto \hat f_t(\omega,y,\eta_t)]$ and similarly for $D_2\hat f_t(\omega,\eta_t)$, $\hat f'_t(\omega,\eta_t)$. We consider them as processes taking values in $\mathcal B_b(W)$.
	Applying mean value theorem, \eqref{est.normDkg} and \eqref{con.gammamholder}, we have $|\hat f_t(\bar \xi)-\hat f_t(\xi)|_\infty\le |\hat f_t|_\gamma\|\bar \xi- \xi\|'_p$ for all \( \xi,\bar \xi\in  L'_q\).
	Combining with the triangle inequality
	$|(f\comp \eta)_t-(f\comp \eta)_s|_\infty
	\le |\delta\hat f_{s,t}(\eta_t)|_\infty+|\hat f_s(\eta_t)-\hat f_s(\eta_s)|_\infty$,
	we have 
	\[
	\||\delta (f\comp\eta)_{s,t}|_\infty |\mathcal F_s\|_m\le \||\delta \hat f_{s,t}|_\infty |\mathcal F_s\|_m +\||\hat f_s|_\gamma \|_n \|\|\delta \eta_{s,t}\|'_p|\cff_s\|_m \,,
	\]
	showing in particular that
	\(  \bk{\delta (f\comp\eta)}_{\beta;m,n}\lesssim \bk{\delta \hat f}_{\beta;m,n}+ M\|\hat f
	\|_{\gamma;n} \). 
	
	Similarly, by triangle inequality, 
	\begin{align*}
		|(f\comp\eta)'_t-(f\comp\eta)'_s|_{\infty}
		&\le |\delta D_2\hat f_{s,t}(\eta_t)[\eta'_t]|_{\infty} + |D_2\hat f_s(\eta_t)[\eta'_t]-D_2\hat f_s( \eta_s)[\eta'_t]|_{\infty}
		\\
		&\quad+|D_2\hat f_s(\eta_s)[\delta\eta'_{s,t} ]|_{\infty}
		+|\delta \hat f'_{s,t}(\eta_t)|_{\infty} +|\hat f'_s(\eta_t)-\hat f'_s(\eta_s)|_{\infty}
		\\&
		\le |\delta D_2\hat f_{s,t}|_\infty\|\eta'_t\|'_p+|\hat f_s|_{\gamma}(\|\delta \eta_{s,t}\|'_p)^{\gamma-1}\|\eta'_t\|'_p 
		\\&\quad+|\hat f_s|_\gamma\|\delta \eta'_{s,t}\|'_p+|\delta\hat f'_{s,t}|_\infty+|\hat f'_s|_{\gamma-1}(\|\delta \eta'_{s,t}\|'_p)^{\gamma-1}.
	\end{align*}
	Using $\C^\gamma_{b,p,q}$-regularity, we arrive at
	\begin{multline*}
		\|\||\delta (f\comp\eta)'_{s,t}|_\infty |\cff_s\|_m\|_n
		\lesssim  (\|\|\eta'_t\|'_p\|_\infty\bk{\delta D_2\hat f}_{\beta';m,n}+\bk{\delta \hat f'}_{\beta';m,n})(t-s)^{\beta'} 
		\\+\left(\|\|\delta \eta\|'_p\|_{\beta;m,\infty}^{\gamma-1} \|\|\eta'_t\|'_p\|_\infty+\|\|\delta \eta'\|'_p\|_{\beta';m,\infty}+\|\|\delta \eta\|'_p\|_{\beta;m,\infty}^{\gamma-1}\right)(t-s)^{\beta''},
	\end{multline*}		
	which by \eqref{con.eta1new} yields the estimate:
	\[
	\bk{\delta ( f\comp\eta)'}_{\beta'';m,n}\lesssim (1+M)^\gamma.
	\]
	The bound $\bk{\delta D(f\comp\eta)}_{\beta'';m,\infty}\lesssim (1+M)^\gamma$ is verified analogously. 
	To estimate $R^{f\comp\eta}_{s,t}=\delta( f\comp\eta)_{s,t}- (f\comp\eta)'_s\delta X_{s,t}$, we write
	\begin{equation}\label{id.Rg}
		\begin{aligned}
			R^{f\comp\eta}_{s,t}
			&= \hat f_s(\eta_t)-\hat f_s(\eta_s)-D_2\hat f_s(\eta_s)[\eta'_s]\delta X_{s,t}
			\\&\quad+\hat f_t(\eta_s)-\hat f_s(\eta_s)-\hat f'_s(\eta_s)\delta X_{s,t}
			\\&\quad+\hat f_t(\eta_t)-\hat f_t(\eta_s)-\hat f_s(\eta_t)+\hat f_s(\eta_s)
			\\&=R^{\hat f_s(\eta)}_{s,t}+R^{\hat f}_{s,t}(\eta_s)+(\delta \hat f_{s,t}(\eta_t)-\delta \hat f_{s,t}(\eta_s))
		\end{aligned}
	\end{equation}
	and treat each term separately.
	Using mean value theorem and \cref{lem.lawDg}, we have for the first term
	\begin{align*}
		R^{\hat f_s(\eta)}_{s,t}
		&=\hat f_s(\eta_t)-\hat f_s(\eta_s)-D_2 \hat f_s(\eta_s)[\delta \eta_{s,t}]+ D_2 \hat f_s(\eta_s)[R^\eta_{s,t}]
		\\&=\int_0^1[D_2\hat f_s(\theta \eta_t+(1- \theta)\eta_s)-D_2\hat f_s(\eta_s)][\delta \eta_{s,t}]d \theta+D_2 \hat f_s(\eta_s)[\E'_s R^\eta_{s,t}].
	\end{align*}		
	From \cref{lem:operator}, linearity and \cref{lem:prod}, we find that 
	\[
	 \E_sD_2\hat f_s(\eta_s)[\E'_s R^\eta_{s,t}]
	 =D_2\hat f_s(\eta_s)[\E_s\E'_s R^\eta_{s,t}]
	 =D_2\hat f_s(\eta_s)[\E\otimes \E'( R^\eta_{s,t}|\mathcal F_s\otimes\mathcal F'_s)].
	\]
Therefore:
	\begin{align*}
		|\E_s R^{\hat f_s(\eta)}_{s,t}|_\infty \lesssim \|\|\delta \eta\|'_p\|_{\beta;m,\infty}^{\gamma}|t-s|^{\gamma \beta}+\|\|(\E\otimes \E')_\bigcdot R^\eta\|'_p\|_{\beta+\beta';\infty}|t-s|^{\beta+\beta'},
	\end{align*}
	showing, indeed, that 
	\[
	\bk{R^{\hat f_s(\eta)}}_{\beta+\beta'';n}\lesssim(1+M)^\gamma\,.
	\]
	For the term $R^{\hat f}_{s,t}(\eta_s)$ in \eqref{id.Rg}, we have
	\begin{align*}
		\||\E_s R^{\hat f}_{s,t}(\eta_s)|_\infty \|_n
		\le\bk{\E_\bigcdot R^{\hat f}}_{\beta+\beta';n}(t-s)^{\beta+\beta'}.
	\end{align*}
	For the last term in \eqref{id.Rg}, we start from the Lipschitz estimate 
	\begin{align*}
		\sup_{y\in W}|\delta \hat f_{s,t}(y,\eta_t)-\delta \hat f_{s,t}(y,\eta_s) |\le\sup_{(y,\xi)\in W\times L_q(W)} |\delta D\hat f_{s,t}(y,\xi)|\|\delta \eta_{s,t}\|'_p.
	\end{align*}
	Consequently, using Cauchy-Schwarz inequality and \eqref{con.eta1new}, we have
	\begin{align*}
		\E_s|\delta \hat f_{s,t}(\eta_t)-\delta \hat f_{s,t}(\eta_s) |_\infty\le (\||\delta D_2\hat f_{s,t}|_\infty|\cff_s\|_2) (\|\|\delta\eta_{s,t}\|_p'|\cff_s\|_2 )
	\end{align*}
and thus:
\[
\|\E_s|\delta \hat f_{s,t}(\eta_t)-\delta \hat f_{s,t}(\eta_s) |_\infty\|_n
\le \bk{\delta D\hat f}_{\beta';m,n}M (t-s)^{\beta+\beta'}.
\]
	Putting these estimates into \eqref{id.Rg} shows the desired estimate.
	This concludes the proof.
\end{proof}

The next result establishes stability for the map $(f,f',\eta,\eta')\mapsto (f\comp\eta, (f\comp\eta)')$ and is an analogue of \cite[Prop. 3.25]{FHL21}. 
\begin{lemma}[Stability]
	\label{lem.compose2}
	Consider $\gamma>2$; $\alpha\in(\frac13,\frac12]$; $\beta,\beta'\in[0,1]$.
	Let $X,\bar X$ be two $\alpha$-H\"older paths.
	Let \((\eta,\eta')\), $(\bar \eta,\bar \eta')$  be stochastic controlled rough paths in \(\mathbf{D}^{\beta,\beta'}_{X}L_{m,\infty}([0,T];\W\otimes\W';W)\) and \(\mathbf{D}^{\beta,\beta'}_{\bar X}L_{m,\infty}([0,T];\W\otimes\W';W)\) respectively such that for every $t\in[0,T]$, $\P$-a.s.\ $\omega$,
	\[
	\eta_t(\omega,\cdot),\bar\eta_t(\omega,\cdot)\in L'_q \quad\text{and}\quad
	\eta_t(\omega,\cdot)-\bar\eta_t(\omega,\cdot)\in L'_p \,.
	\] 
	Assume furthermore that a finite constant \( M\) exists such that 
	\begin{equation}\label{con.YboundedM}
		\|\|\delta \xi\|'_p\|_{\beta; m,\infty}+\|\|\xi'\|'_p\|_{\beta';m,\infty}+\|\|(\E\otimes\E')_\bigcdot R^\xi\|'_p\|_{\beta+\beta';\infty} 
		\le M\quad \text{for each}\enskip\xi\in \{\eta,\bar\eta\}.
	\end{equation}

	Let $(f,f'), (\bar f,\bar f')$ be measure-dependent coefficients as in \eqref{measure-dependent-lemma} with Lions-lifts subject to
	\[
	(\hat f,\hat f')\in\mathbf{D}^{\beta,\beta}_XL_{m,\infty}\C^\gamma_{b,p,q}\,,
	\quad \quad 
	(D\hat f,D\hat f')\in\mathbf{D}^{\beta,\beta'}_XL_{m,\infty}\C^{\gamma-1}_{b,p,q}
	\]
	while 
	\[
	(\hat{\bar f},\hat{\bar f}')\in\mathbf{D}^{\beta,\beta'}_{\bar X}L_{m,\infty}\C^{\gamma-1}_{b,p,q}\,.
	\]
	
	\textbf{Spatial regularity.} 
	For every $t\in [0,T]$, we have
	\begin{align}
		\label{sup_f_eta}
		&\||f\comp\eta_t-\bar f\comp\bar\eta_t|_{\gamma-1}\|_m
		\lesssim \||\hat f_t-\hat {\bar f}_t|_{\gamma-1}\|_m+  \|\|\eta_t-\bar\eta_t\|'_p\wedge1\|_m\,,
		\\
		\label{sup_f_eta_prime}
		&\|| (f\comp\eta)'_t-(\bar f\comp\bar\eta)'_t|_{\gamma-2} \|_m
		\lesssim \||\hat f'_t-\hat {\bar f}'_t|_{\gamma-2}\|_m+
		\|\|\eta_t -\bar\eta_t\|'_p\wedge1\|_m+\|\|\eta'_t-\bar \eta'_t\|'_p\|_m\,.
	\end{align}
	
	\textbf{Temporal regularity.}  Let $\kappa\in[0, \beta]$ and $\kappa'\in[0, \min(\beta,\beta',(\gamma-2)\beta)]$ be some fixed numbers and put
\( F:=\bk{(\hat f-\hat{\bar f},\hat f'-\hat {\bar f}')}_{X,\bar X;\kappa,\kappa';m}+\|(\hat f-\hat{\bar f},\hat f'-\hat{\bar f}')\|_{\gamma-1;m} \). We have for every $(s,t)\in \Delta(0,T)$:
	\begin{align}
		&\| |\delta (f\comp\eta)_{s,t}-\delta (\bar f\comp\bar\eta)_{s,t}|_{\infty}\|_m
		\lesssim |t-s|^{\kappa}\left( F+\|\|\eta_s-\bar \eta_s\|'_p\wedge1\|_m \right)  
		\nonumber\\&\quad \quad \quad \quad \quad \quad \quad \quad \quad  
		\quad \quad \quad \quad 
		+ \| \|\delta \eta _{s,t}- \delta\bar\eta_{s,t}\|'_p\|_m\,,
		\label{delta_f_Y}
		\\[0.5em]
		&\||\delta D(f\comp\eta)_{s,t}-\delta D(\bar f\comp\bar\eta)_{s,t}|_{\infty}\|_m
		\lesssim|t-s|^{\kappa'}(F+\|\|\eta_s-\bar \eta_s\|'_p\wedge1\|_m) 
		\nonumber\\&\quad \quad \quad \quad \quad \quad \quad 
		\quad \quad \quad \quad \quad \quad \quad \quad \quad 
		+ \|\|\delta \eta _{s,t}- \delta\bar\eta_{s,t}\|'_p\|_m\,,
		\label{delta_Df_Y}
		\\[0.5em]
		&\||\delta  (f\comp\eta)'_{s,t}-\delta (\bar f\comp\bar\eta)'_{s,t}|_{\infty}\|_m
		\lesssim
		|t-s|^{\kappa'}(F+\|\|\eta_s-\bar \eta_s\|'_p\wedge1\|_m+\|\|\eta'_s-\bar \eta'_s\|'_p\|_m)
		\nonumber\\&\quad\quad \quad \quad \quad \quad \quad 
		\quad \quad \quad \quad 
		+\|\|\delta \eta_{s,t}-\delta \bar\eta_{s,t}\|'_p\|_m+\|\|\delta \eta'_{s,t}-\delta \bar\eta'_{s,t}\|'_p\|_m\,,
		\label{delta_f_Y_prime}
		\\[0.5em]
		&\big\||\E_sR^{f\comp\eta}_{s,t}- \E_s \bar R^{\bar f\comp\bar\eta}_{s,t}|_\infty\|_m
		\lesssim |t-s|^{\kappa+\kappa'}(F+\|\|\eta_s-\bar \eta_s\|'_p\wedge1\|_m)
		\nonumber\\&\quad \quad \quad \quad \quad \quad \quad 
		+|t-s|^{\kappa'}
		\|\|\delta \eta_{s,t}-\delta \bar \eta_{s,t}\|'_p\|_m
		+\|\|\E_s\otimes \E'_s(R^\eta_{s,t}-\bar R^{\bar \eta}_{s,t})\|'_p\|_m\,.
		\label{E_R_f_Y}
	\end{align}
	The above implied constants are deterministic and depend only upon the quantities \( M \), $T, \alpha,\gamma$, \( \|(\hat f,\hat f')\|_{\gamma;\infty} \), \( \llbracket (\hat f,\hat f')\rrbracket_{X;\beta,\beta;m,\infty} \) and \(  \llbracket (D\hat f,D\hat f')\rrbracket_{X;\beta,\beta';m,\infty}\).
\end{lemma}
 Before presenting the proof, we state an important consequence of \cref{lem.compose2} which is useful for the proofs of the main results.
\begin{corollary}\label{cor.expo}
	Let \(\gamma,\alpha,\beta,\beta',\kappa,\kappa'\),  \(X,\bar X\), $(f,f')$ and $(\bar f, \bar f')$ be as in \cref{lem.compose2}. 
	%
	Let   $(\zeta,\zeta')$ and $(\bar \zeta,\bar \zeta')$ be stochastic controlled processes in \( \mathbf D_{X}^{\beta,\beta'}L_{p}([0,T],\W;\bar W)\) and $\mathbf D_{\bar X}^{\beta,\beta'}L_{p}([0,T],\W;\bar W)$ respectively.
	Then
	\begin{align*}
		\|(f^\zeta-\bar f^{\bar \zeta},(f^\zeta)'-(\bar f^{\bar \zeta})')\|_{\gamma-1;m}\lesssim\|(f-\bar f,f'-\bar f')\|_{\gamma-1;m}+\sup_{t\in[0,T]}(\|\zeta_t-\bar \zeta_t\|_p+\|\zeta'_t-\bar \zeta'_t\|_p)
	\end{align*}
	and
	\begin{align*}
		\bk{f^\zeta,(f^\zeta)';\bar f^{\bar \zeta},(\bar f^{\bar \zeta})'}_{X,\bar X;\kappa,\kappa';m}
		&\les \bk{(f,f';\bar f,\bar f')}_{X,\bar X;\kappa,\kappa';m}+\|(f-\bar f,f'-\bar f')\|_{\gamma-1;m}
		\\&\quad+ \|\zeta,\zeta';\bar \zeta,\bar \zeta'\|_{X,\bar X;\beta,\beta';p}\,.
	\end{align*}
\end{corollary}
\begin{proof}
	This result is a direct consequence of \cref{lem.compose2} and the relation \eqref{id.expiscomp}. Note that condition \eqref{con.YboundedM} is verified by the assumed properties of $(\zeta,\zeta')$ and $(\bar \zeta,\bar \zeta')$.
\end{proof}

\begin{proof}[Proof of \cref{lem.compose2}]
	In this current proof,
	we follow the functional-analytic standpoint laid out in \cref{def.scvec}. That is to say, given $(t,\omega,\xi)\in [0,T]\times\Omega\times L_q'$ 
	we write \( \hat f_t(\omega,\xi)\) for \([y\mapsto \hat f_t(\omega,y,\xi)] \)
	and regard \( f \) as a stochastic process taking values in $\mathcal Y=(\mathcal B_b(W),|\cdot|_{\infty})$. We adopt a similar view for \( \hat f' \)
	and introduce similarly the (vector-valued) stochastic processes
	\[
	Z_t(\omega)=[y\mapsto f\comp\eta_t(\omega,y)], \quad
	Z'_t(\omega)=[y\mapsto(f\comp\eta_t)'(\omega,y)].
	\]
	The pair \(  (\bar Z,\bar Z')\) is defined accordingly. 
	For computational ease, we will make use the shorthands 
	$g=f-\bar f$, \( \zeta=\eta-\bar \eta \) and \( \zeta'=\eta'-\bar \eta'\).
	The following elementary estimates will be used implicitly:
	for every non-negative numbers $a,b,c$, it holds
	$(a+b)\wedge1\le a\wedge1+b\wedge1$ and $(ab)\wedge c\le (a+c)(b\wedge1)$.
	\smallskip
	
	The first two bounds
	are obtained as direct consequences of the mean-value theorem and the assumed regularity on $f,f'$. Indeed,
	let us assume without loss of generality that $\gamma\le 3$.
	We write (using \eqref{id.mvt})
	\begin{align}\label{tmp.mvtebe}
		Z_t-\bar Z_t
		=(g\comp\bar \eta)_t+[(f\comp \eta)_t-(f\comp \bar \eta)_t]
		=(g\comp \bar \eta)_t+ \int_0^1D_2\hat f_t(\theta \eta_t+(1-\theta)\bar \eta_t)[\zeta_t]d\theta
	\end{align}
	and
	\begin{align}
		&Z'_t-\bar Z'_t
		=(g\comp \bar\eta)'_t+[(f\comp \eta)'_t-(f\comp \bar \eta)'_t]
		\nonumber
		\\&=(g\comp \bar\eta)'_t+ \int_0^1\frac d{d \theta}\left(D_2\hat f_t(\theta \eta_t+(1- \theta)\bar \eta_t)[\theta \eta'_t+(1- \theta)\bar \eta'_t]+\hat f_t'(\theta \eta_t+(1- \theta)\bar \eta_t)\right)d \theta
		\nonumber\\
		&=(g\comp \bar\eta)'_t+\int_0^1 D^2_2\hat f_t(\theta \eta_t+(1- \theta)\bar \eta_t)[\theta \eta'_t+(1- \theta)\bar \eta'_t,\zeta_t]d \theta
		\nonumber\\&\quad
		+\int_0^1 (D_2\hat f_t+D_2\hat f'_t)(\theta \eta_t+(1- \theta)\bar \eta_t)[\zeta'_t]d \theta.\label{MVT2}
	\end{align}
	From \eqref{tmp.mvtebe},
	 using boundedness of $D_2\hat f_t,\hat f_t$ 
	we obtain that
	\begin{align*}
		|Z_t-\bar Z_t|_{\gamma-1}\le
		|\hat g_t|_{\C^{\gamma-1}_{b,p,q}}+
		2\||\hat f_t|_{\C^{\gamma}_{b,p,q}}\|_\infty (\|\zeta_t\|_p'\wedge1).
	\end{align*}
	This implies \eqref{sup_f_eta}. Estimate \eqref{sup_f_eta_prime} is obtained analogously from \eqref{MVT2} and details have been omitted.
	\smallskip

	\textit{Step 1: We show \eqref{delta_f_Y}-\eqref{delta_Df_Y}.}	
	From \eqref{tmp.mvtebe}, we have
	\begin{align*}
		\delta  Z_{s,t}-\delta \bar Z_{s,t}
		= 
		\delta(g\comp \bar\eta)_{s,t}+\delta(f\comp \eta- f\comp\bar \eta)_{s,t}
	\end{align*}
	and	
	\begin{equation*}
		\begin{aligned}
			\delta(f\comp \eta- f\comp\bar \eta)_{s,t}
			&=
			\int_0^1 D_2\hat f_t(\theta \eta_t+(1- \theta)\bar \eta_t)[\delta \zeta_{s,t}] d \theta
			+\int_0^1 \delta D_2\hat f_{s,t}( \theta \eta_t+(1- \theta)\bar \eta_t)[\zeta_s]d \theta
			\\&\quad+\int_0^1[ D_2\hat f_s(\theta \eta_t+(1- \theta)\bar \eta_t)- D_2\hat f_s(\theta \eta_s+(1- \theta)\bar \eta_s)][\zeta_s] d \theta\,.
		\end{aligned}
	\end{equation*}
	
	Using triangle inequality and conditioning (as in the proof of \cref{prop.newcvec}), we have
	\begin{align*}
		|\delta(g\comp \bar\eta)_{s,t}|_\infty \le |\delta  \hat g_{s,t}|_\infty +|\hat g_s|_{\C^1_{b,p,q}} \|\delta \eta_{s,t}\|_p'.
	\end{align*}
	To estimate the remaining terms, we use regularity of $f$ (i.e. \eqref{con.gammamholder}, \eqref{est.normDkg} and \eqref{est.normdgammag}). We obtain \( \mathbb P \)-almost surely that
	\begin{align*}
		|\delta(f\comp \eta- f\comp\bar \eta)_{s,t}|_\infty 
		\le |D_2\hat f_t|_\infty \|\delta \zeta_{s,t}\|'_p
		+\left(|\delta D\hat f_{s,t}|_{\infty}
		+|D_2\hat f_s|_{\C^1_{b,p,q}}(\|\delta \eta_{s,t}\|'_{p}
		+\|\delta \bar \eta_{s,t}\|'_{p})\right)\|\zeta_s\|'_p\,.
	\end{align*}	
	On the other hand, it is obvious that 
	$|\delta(f\comp \eta- f\comp\bar \eta)_{s,t}|_\infty\lesssim|\delta \hat f_{s,t}|_\infty+\|\delta \eta_{s,t}\|'_p+\|\delta \bar\eta_{s,t}\|'_p $, and hence,
	\begin{align*}
		|\delta(f\comp \eta- f\comp\bar \eta)_{s,t}|_\infty 
		\lesssim\|\delta \zeta_{s,t}\|'_p
		+\left(|\delta \hat f_{s,t}|_{\infty}+|\delta D\hat f_{s,t}|_{\infty}
		+\|\delta \eta_{s,t}\|'_{p}
		+\|\delta \bar \eta_{s,t}\|'_{p}\right) (\|\zeta_s\|'_p\wedge1)\,.
	\end{align*}	
	It follows that
	\begin{align*}
		|\delta Z_{s,t}-\delta \bar Z_{s,t}|_\infty 
		&\lesssim
		|\delta \hat g_{s,t}|_\infty +|\hat g_s|_{\C^1_{b,p,q}} \|\delta \eta_{s,t}\|_p'
		\\&\quad+
		\|\delta \zeta_{s,t}\|'_p
		+
		\left(|\delta \hat f_{s,t}|_{\infty}+|\delta D\hat f_{s,t}|_{\infty}
		+\|\delta \eta_{s,t}\|'_{p}
		+\|\delta \bar \eta_{s,t}\|'_{p}\right) (\|\zeta_s\|'_p\wedge1)\,.
	\end{align*}	
	We now take moment on both sides, using conditioning and \eqref{con.YboundedM} to obtain \eqref{delta_f_Y}.
	Similarly, we have
	\begin{multline*}
		|\delta D Z_{s,t}-\delta D \bar Z_{s,t}|_\infty 
		\lesssim
		|\delta  D_1\hat g_{s,t}|_\infty +|D_1\hat g_s|_{\C^{\gamma-2}_{b,p,q}} (\|\delta \eta_{s,t}\|_p')^{\gamma-2}+|D_2D_1\hat f_t|_\infty \|\delta \zeta_{s,t}\|'_p
		\\
		+[|\delta D_2\hat f_{s,t}|_{\infty}+|\delta D_2D_1\hat f_{s,t}|_{\infty}
		+(\|\delta \eta_{s,t}\|'_{p})^{\gamma-2}
		+(\|\delta \bar \eta_{s,t}\|'_{p})^{\gamma-2})](\|\zeta_s\|'_p\wedge1)\,,
	\end{multline*}	
	which yields  \eqref{delta_Df_Y}.\smallskip

	\textit{Step 2: We show \eqref{delta_f_Y_prime}.}
	We hinge on \eqref{MVT2}. From the identity
	\begin{align*}
		\delta(g\comp \bar \eta)'_{s,t}
		&= (\delta D_2\hat g_{s,t})(\bar \eta_t)[\bar \eta'_t]
		+D_2\hat g_s(\bar \eta_t)[\delta\bar \eta'_{s,t}]
		+(D_2\hat g_s(\bar \eta_t)-D_2\hat g_s(\bar \eta_s))[\bar \eta'_s]
		\\&\quad+(\delta \hat g'_{s,t})(\bar \eta_t)
		+\hat g'_s(\bar \eta_t)-\hat g'_s(\bar \eta_s),
	\end{align*}
	applying \eqref{con.gammamholder}, it is clear that
	\begin{align*}
		|\delta(g\comp \bar \eta)'_{s,t}|_\infty
		&\le |\delta D_2\hat g_{s,t}|_{\infty}\|\bar \eta'_t\|'_p+|D_2\hat g_s|_\infty\|\delta\bar \eta'_{s,t}\|'_p+|D_2\hat g_s|_{\C^{\gamma-2}_{b,p,q}}(\|\delta\bar \eta_{s,t}\|'_p)^{\gamma-2}\|\bar \eta'_s\|'_p
		\\&\quad+ |\delta \hat g'_{s,t}|_\infty+|\hat g'_s|_{\C^{\gamma-2}_{b,p,q}}(\|\delta \bar \eta_{s,t}\|'_p)^{\gamma-2}.
	\end{align*}
	To obtain \eqref{delta_f_Y_prime}, it suffices to show that
	\begin{multline*}
		|\delta(f\comp \eta)_{s,t}- \delta(f\comp\bar \eta)_{s,t}|_\infty
		\lesssim
		\bigg(|\delta D_2^2\hat f_{s,t}|_{\infty}+|\delta D_2\hat f_{s,t}|_\infty+|\delta\hat f'_{s,t}|_\infty
		\\
		+  (\|\delta \eta_{s,t}\|'_p+\|\delta \bar \eta_{s,t}\|'_p)^{\gamma-2}
		+ (\|\delta \eta'_{s,t}\|'_p+\|\delta\bar \eta'_{s,t}\|'_p)\bigg) (\|\zeta_s\|'_p\wedge1)
		\\
		+\left (|\delta D_2\hat f_{s,t}|_{\infty}+|\delta D_2 \hat f'_{s,t}|_{\infty}+(\|\delta \eta_{s,t}\|'_p+\|\delta\bar \eta_{s,t}\|'_p)^{\gamma-2}\right )\|\zeta'_s\|'_p
		+ \|\delta\zeta'_{s,t}\|'_p.
	\end{multline*}
	We note that			
	\begin{multline*}
		\Big|\delta\Big(\int_0^1 D^2_2\hat f(\theta \eta+(1-\theta)\bar \eta)[\theta \eta'+(1- \theta)\bar \eta',\zeta]d \theta\Big)_{s,t}\Big|_\infty 
		\\
		\lesssim 
		|\delta D_2^2\hat f_{s,t}|_{\infty}(\|\eta'_s\|'_p+\|\bar \eta'_s\|'_p)\|\zeta_s\|'_p
		+ |D_2^2\hat f_t|_{\C^{\gamma-2}_{b,p,q}} (\|\delta \eta_{s,t}\|'_p+\|\delta \bar \eta_{s,t}\|'_p)^{\gamma-2}(\|\eta'_s\|'_p+\|\bar \eta'_s\|'_p)\|\zeta_s\|'_p
		\\
		+|D_2^2\hat f_t|_\infty (\|\delta \eta'_{s,t}\|'_p+\|\delta\bar \eta'_{s,t}\|'_p)\|\zeta_s\|'_p
		+|D_2^2\hat f_t|_\infty(\|\eta'_t\|'_p+\|\bar \eta'_t\|'_p)\|\delta\zeta_{s,t}\|'_p
	\end{multline*}
	and
	\begin{align*}
		&\Big|\delta\Big(\int_0^1 (D_2\hat f+D_2\hat f')(\theta \eta+(1- \theta)\bar \eta)[\zeta']d \theta\Big)_{s,t}\Big|_\infty 
		\lesssim \left (|\delta D_2\hat f_{s,t}|_{\infty}+|\delta D_2 \hat f'_{s,t}|_{\infty}\right )\|\zeta'_s\|'_p
		\\&\quad+  \left (| D_2\hat f_{t}|_{\C^{\gamma-2}_{b,p,q}}+| D_2\hat f'_{t}|_{\C^{\gamma-2}_{b,p,q}}\right )(\|\delta \eta_{s,t}\|'_p+\|\delta\bar \eta_{s,t}\|'_p)^{\gamma-2}\|\zeta'_s\|'_p
		+\left (| D_2\hat f_{t}|_{\infty}+| D_2\hat f'_{t}|_{\infty}\right )\|\delta\zeta'_{s,t}\|'_p.
	\end{align*}
	Hence, by \eqref{MVT2} and \eqref{con.YboundedM}, we have
	\begin{multline*}
		|\delta(f\comp \eta)'_{s,t}- \delta(f\comp\bar \eta)'_{s,t}|_\infty
		\lesssim \left(|\delta D_2^2\hat f_{s,t}|_{\infty}
		+  (\|\delta \eta_{s,t}\|'_p+\|\delta \bar \eta_{s,t}\|'_p)^{\gamma-2}
		+ (\|\delta \eta'_{s,t}\|'_p+\|\delta\bar \eta'_{s,t}\|'_p)\right) \|\zeta_s\|'_p
		\\+\left (|\delta D_2\hat f_{s,t}|_{\infty}+|\delta D_2 \hat f'_{s,t}|_{\infty}+(\|\delta \eta_{s,t}\|'_p+\|\delta\bar \eta_{s,t}\|'_p)^{\gamma-2}\right )\|\zeta'_s\|'_p
		+ \|\delta\zeta'_{s,t}\|'_p.
	\end{multline*}
	On the other hand, using \eqref{con.YboundedM}, it is evident that
	\begin{align*}
		|\delta(f\comp \eta)'_{s,t}- \delta(f\comp\bar \eta)'_{s,t}|_\infty
		&\lesssim |\delta D_2\hat f_{s,t}|_\infty+|\delta\hat f'_{s,t}|_\infty
		\\&\quad+(\|\delta \eta_{s,t}\|'_p+\|\delta \bar \eta_{s,t}\|'_p)^{\gamma-2}+(\|\delta \eta'_{s,t}\|'_p+\|\delta\bar \eta'_{s,t}\|'_p).
	\end{align*}
	Combining the previous two inequalities, we obtain the claimed inequality.
	\smallskip
	
	\textit{Step 3: We show \eqref{E_R_f_Y}.} Similar to \eqref{id.Rg}, we write
	\begin{equation}\label{id.Rfeta}
		\begin{aligned}
			R^{f\comp \eta}_{s,t}
			=Tf_s(\eta_s,\eta_t) +D_2\hat f_s(\eta_s)[R^{\eta}_{s,t}]
			+R^{\hat f}_{s,t}(\eta_s)
			+(\delta\hat f_{s,t}(\eta_t)-\delta\hat f_{s,t}(\eta_s))
		\end{aligned}
	\end{equation}
	where
	\begin{align*}
		Th(\xi,\eta)=\hat h(\eta)-\hat h(\xi)-D_2\hat h(\xi)[\eta- \xi].
	\end{align*}
	We decompose $R^{\bar f\comp \bar \eta}_{s,t}$ in an analogous way.
	We estimate separately the differences of the corresponding terms on the right-hand sides of the two decompositions.
	
	We have
	\begin{align*}
		Tf_s(\eta_s,\eta_t)-T\bar f_s(\bar \eta_s,\bar \eta_t)=Tg_s(\bar \eta_s,\bar \eta_t)+Tf_s(\eta_s,\eta_t)-Tf_s(\bar \eta_s,\bar \eta_t).
	\end{align*}
	By Taylor's expansion, it is evident that
	\begin{align*}
		|Tg_s(\bar \eta_s,\bar \eta_t)|_\infty\le |\hat g_s|_{\C^{\gamma-1}_{b,p,q}}(\|\delta\bar \eta_{s,t}\|'_p)^{\gamma-1},
	\end{align*}
	and hence, using \eqref{con.YboundedM},
	\begin{align*}
		\||\E_sTg_s(\bar \eta_s,\bar \eta_t)|_\infty\|_m\le \||\hat g_s|_{\C^{\gamma-1}_{b,p,q}}\|_m(t-s)^{(\gamma-1)\beta}.
	\end{align*}
	Next, we put $\zeta^\theta=\theta \eta+(1- \theta)\bar \eta$ and apply mean value theorem to get that
	\begin{align*}
		&Tf_s(\eta_s,\eta_t)-Tf_s(\bar \eta_s,\bar \eta_t)
		=\int_0^1\frac{d}{d \theta}\left(\hat f_s(\zeta^\theta_t)-\hat f_s(\zeta^\theta_s)-D_2\hat f_s(\zeta^\theta_s)[\delta \zeta^\theta_{s,t}] \right)d \theta
		\\&=\int_0^1\left( D_2\hat f_s(\zeta^\theta_t)[\zeta_t]-D_2\hat f_s(\zeta^\theta_s)[\zeta_s]-D_2\hat f_s(\zeta^\theta_s)[\delta\zeta_{s,t}]-D_2^2\hat f_s(\zeta^\theta_s)[\zeta_s,\delta \zeta^\theta_{s,t}]\right)d \theta
		\\&=\int_0^1\left(D_2\hat f_s(\zeta^\theta_t)[\zeta_s]-D_2\hat f_s(\zeta^\theta_s)[\zeta_s]-D_2^2\hat f_s(\zeta^\theta_s)[\zeta_s,\delta \zeta^\theta_{s,t}]\right)d \theta
		\\&\quad+\int_0^1\left( D_2\hat f_s(\zeta^\theta_t)[\delta\zeta_{s,t}]-D_2\hat f_s(\zeta^\theta_s)[\delta\zeta_{s,t}]\right)d \theta.
	\end{align*}
	Using regularity of $D_2\hat f_s$, we get that
	\begin{multline*}
		|Tf_s(\eta_s,\eta_t)-Tf_s(\bar \eta_s,\bar \eta_t)|_\infty
		\les |D_2\hat f_s|_{\C^{\gamma-1}_{b,p,q}}\int_0^1\left(\|\zeta_s\|'_p(\|\delta \zeta^\theta_{s,t}\|'_p)^{\gamma-1}+\|\delta \zeta^\theta_{s,t}\|'_p\|\delta \zeta_{s,t}\|'_p \right)d \theta.
		\\\les \left((\|\delta \eta_{s,t}\|'_p)^{\gamma-1}+\|\delta \bar\eta_{s,t}\|'_p)^{\gamma-1} \right)\|\zeta_s\|'_p+\left(\|\delta \eta_{s,t}\|'_p+\|\delta\bar \eta_{s,t}\|'_p\right)\|\delta \zeta_{s,t}\|'_p.
	\end{multline*}
	On the other hand, we also have
	\begin{align*}
		|Tf_s(\eta_s,\eta_t)-Tf_s(\bar \eta_s,\bar \eta_t)|_\infty
		\les|\hat f_s|_{\C^{\gamma-1}_{b,p,q}}\left( (\|\delta \eta_{s,t}\|'_p)^{\gamma-1}+(\|\delta \bar\eta_{s,t}\|'_p)^{\gamma-1}\right).
	\end{align*}
	Combining the two estimates yields that
	\begin{align*}
		|Tf_s(\eta_s,\eta_t)-Tf_s(\bar \eta_s,\bar \eta_t)|_\infty
		&\lesssim \left((\|\delta \eta_{s,t}\|'_p)^{\gamma-1}+\|\delta \bar\eta_{s,t}\|'_p)^{\gamma-1} \right)(\|\zeta_s\|'_p\wedge1)
		\\&\quad+\left(\|\delta \eta_{s,t}\|'_p	+\|\delta\bar \eta_{s,t}\|'_p\right)\|\delta \zeta_{s,t}\|'_p.
	\end{align*}	
	Thus, using \eqref{con.YboundedM}, we have
	\begin{align*}
		\||\E_s(Tf_s(\eta_s,\eta_t)-Tf_s(\bar \eta_s,\bar \eta_t))|_\infty\|_m
		\lesssim \|\|\zeta_s\|'_p\wedge1\|_m(t-s)^{(\gamma-1)\beta}+\|\|\delta \zeta_{s,t}\|'_p\|_m(t-s)^{\beta}.
	\end{align*}
	It follows that 
	\begin{multline*}
		\||\E_s(Tf_s(\eta_s,\eta_t)-T\bar f_s(\bar \eta_s,\bar \eta_t))|_\infty\|_m
		\\\lesssim (\||\hat g_s|_{\C^{\gamma-1}_{b,p,q}}\|_m+\|\|\zeta_s\|'_p\wedge1\|_m)(t-s)^{(\gamma-1)\beta}+\|\|\delta \zeta_{s,t}\|'_p\|_m(t-s)^\beta.
	\end{multline*}
	
	For the second difference, we apply \cref{lem.lawDg} to see that
	\begin{align*}
		D_2\hat f_s(\eta_s)[R^\eta_{s,t}]-D_2\hat{\bar f}_s(\bar\eta_s)[\bar R^{\bar \eta}_{s,t}]
		&=D_2\hat f_s(\eta_s)[\E'_sR^\eta_{s,t}]-D_2\hat{\bar f}_s(\bar\eta_s)[\E'_s\bar R^{\bar \eta}_{s,t}]
		\\&=D_2\hat g_s(\bar\eta_s)[\E'_s\bar R^{\bar \eta}_{s,t}]
		+\left(D_2\hat f_s(\eta_s)[\E'_sR^\eta_{s,t}]-D_2\hat f_s(\bar\eta_s)[\E'_s\bar R^{\bar \eta}_{s,t}]\right)
	\end{align*}
	and hence using \cref{lem:operator},
	\begin{multline*}
		\E_s\left(D_2\hat f_s(\eta_s)[R^\eta_{s,t}]-D_2\hat{\bar f}_s(\bar\eta_s)[\bar R^{\bar \eta}_{s,t}]\right)
		\\=D_2\hat g_s(\bar\eta_s)[\E_s\E'_s\bar R^{\bar \eta}_{s,t}]
		+\left(D_2\hat f_s(\eta_s)[\E_s\E'_sR^\eta_{s,t}]-D_2\hat f_s(\bar\eta_s)[\E_s\E'_s\bar R^{\bar \eta}_{s,t}]\right).
	\end{multline*}
	Consequently, the regularity of $f,\bar f$ and \cref{lem:prod} imply
	\begin{multline*}
		|\E_s\left(D_2\hat f_s(\eta_s)[R^\eta_{s,t}]-D_2\hat{\bar f}_s(\bar\eta_s)[\bar R^{\bar \eta}_{s,t}]\right)|_\infty
		\lesssim |D_2\hat g_s|_\infty\|(\E\otimes\E')_s\bar R^{\bar \eta}_{s,t}\|'_p
		\\+|D_2\hat f_s|_{\C^1_{b,p,q}}(\|\zeta_s\|'_p\wedge1)\|(\E\otimes\E')_s\bar R^{\bar \eta}_{s,t}\|'_p+|D_2\hat f_s|_\infty\|(\E\otimes\E')_s(R^\eta_{s,t}-\bar R^{\bar \eta}_{s,t})\|'_p.
	\end{multline*}
	Taking into account the condition \eqref{con.YboundedM}, we have
	\begin{multline*}
		\||\E_s(D_2\hat f_s(\eta_s)[R^\eta_{s,t}]-D_2\hat{\bar f}_s(\bar\eta_s)[\bar R^{\bar \eta}_{s,t}])|_\infty\|_m
		\\\lesssim (\||D_2\hat g_s|_\infty \|_m+\|\|\zeta_s\|'_p\wedge1\|_m)(t-s)^{\beta+\beta'}+\|\|(\E\otimes\E')_s(R^\eta_{s,t}-\bar R^{\bar \eta}_{s,t})\|'_p\|_m.
	\end{multline*}
	
	For the third difference, we have
	\begin{align*}
		R^{\hat f}_{s,t}(\eta_s)-R^{\hat{\bar f}}_{s,t}(\bar \eta_s)
		&=(R^{\hat f}_{s,t}-R^{\hat{\bar f}}_{s,t})(\bar\eta_s)
		+R^{\hat f}_{s,t}(\eta_s) -R^{\hat f}_{s,t}(\bar \eta_s).
	\end{align*}
	We note that 
	\begin{align*}
		|\E_s(R^{\hat f}_{s,t}(\eta_s) -R^{\hat f}_{s,t}(\bar \eta_s))|\le 2\bk{\E_\bigcdot R^{\hat f}}_{2\beta;\infty} (t-s)^{2\beta}
	\end{align*}
	and by the mean value theorem that
	\begin{align*}
		|\E_s(R^{\hat f}_{s,t}(\eta_s) -R^{\hat f}_{s,t}(\bar \eta_s))|
		&=\Big|\int_0^1R^{D_2\hat f}_{s,t}(\theta \eta_s+(1- \theta)\bar \eta_s)[\zeta_s]d \theta\Big|
		\\&\le \bk{\E_\bigcdot R^{D_2\hat f}}_{\beta+\beta';\infty}
		(t-s)^{\beta+\beta'}\|\zeta_s\|'_p.
	\end{align*}
	Combining the previous inequalities, we obtain that
	\begin{align*}
		\||\E_s(R^{\hat f}_{s,t}(\eta_s)-R^{{\bar f}}_{s,t}(\bar \eta_s))|_\infty \|_m
		\lesssim \||\E_s(R^{\hat f}_{s,t}-R^{{\bar f}}_{s,t})|_\infty \|_m
		+\|\|\zeta_s\|'_p\wedge1\|_m(t-s)^{\beta+\beta\wedge \beta'}.
	\end{align*}
	
	For the last difference, we write
	\begin{multline*}
		\delta\hat f_{s,t}(\eta_t)-\delta\hat f_{s,t}(\eta_s)
		-(\delta \hat{\bar f}_{s,t}(\bar \eta_t)-\delta \hat{\bar f}_{s,t}(\bar \eta_s))
		\\=\delta\hat g_{s,t}(\bar\eta_t)-\delta\hat g_{s,t}(\bar \eta_s)
		+\left[\delta\hat f_{s,t}(\eta_t)-\delta\hat f_{s,t}(\eta_s)
		-(\delta \hat{f}_{s,t}(\bar \eta_t)-\delta \hat{f}_{s,t}(\bar \eta_s))\right].
	\end{multline*}
	Using the elementary estimate
	\begin{align*}
		|h(a)-h(b)-h(c)+h(d)|&\le |Dh|_{\C^1_{b,p,q}}(\|a-b\|'_p+\|c-d\|'_p)(\|b-d\|'_p\wedge1)
		\\&\quad+|Dh|_{\infty}\|a-b-c+d\|'_p
	\end{align*}
	we have
	\begin{multline*}
		|\delta\hat f_{s,t}(\eta_t)-\delta\hat f_{s,t}(\eta_s)
		-(\delta \hat{\bar f}_{s,t}(\bar \eta_t)-\delta \hat{\bar f}_{s,t}(\bar \eta_s))|_\infty
		\\\lesssim |\delta\hat g_{s,t}|_{\C^1_{b,p,q}}\|\delta \bar \eta_{s,t}\|'_p+|\delta D\hat f_{s,t}|_{\C^1_{b,p,q}}(\|\delta \eta_{s,t}\|'_p+\|\delta\bar \eta_{s,t}\|'_p)(\|\zeta_s\|'_p\wedge1)+|\delta D\hat f_{s,t}|_{\infty}\|\delta \zeta_{s,t}\|'_p.
	\end{multline*}
	Taking into account the regularity of $f$ and \eqref{con.YboundedM}, we obtain that
	\begin{multline*}
		\||\E_s[\delta\hat f_{s,t}(\eta_t)-\delta\hat f_{s,t}(\eta_s)
		-(\delta \hat{\bar f}_{s,t}(\bar \eta_t)-\delta \hat{\bar f}_{s,t}(\bar \eta_s))]|_\infty\|_m
		\\\lesssim \||\delta\hat g_{s,t}|_{\C^1_{b,p,q}}\|_m(t-s)^\beta +\|\|\zeta_s\|'_p\wedge1\|_m(t-s)^{\beta'+\beta}+\|\|\delta \zeta_{s,t}\|'_p\|_m(t-s)^{\beta}.
	\end{multline*}
	
	Summing up the estimates for all the differences, we obtain \eqref{E_R_f_Y}.
\end{proof}

\section{Proof of the main results} 
\label{sec:proofs}
\subsection{Solution theory of MKV equations with rough common noise}
\label{sec:proofs_MV_wp}
The proof of \cref{thm.MV.wellposed} follows a standard two-step argument which allows direct applications of well-posedness and stability results for rough stochastic differential equation \eqref{eq:RSDE} (\cref{thm.fixpoint,thm.stability_precise}). First, we solve the former equation in \eqref{eqn.MVrsde} for each stochastic controlled rough path $(\eta,\eta')$. Without the restriction $(\eta,\eta')=(Y,\hat f(Y,\slashed Y))$, equation \eqref{eqn.MVrsde} is an instance of \eqref{eq:RSDE}. Then, we show that the map which assigns each $(\eta,\eta')$ to the corresponding solution obtained in the previous step is  Lipschitz (more specifically, contractive for \( T\ll1\)). This allows us to apply the Banach fixed-point theorem to obtain a unique solution to \eqref{eqn.MVrsde}.
\ref{sec:proofs_MV_wp}
\begin{proof}[\bf Proof of \cref{thm.MV.wellposed}]
	We choose and fix $\beta,\beta'\in(0,\alpha)$ such that $\beta>\frac13\vee\frac1 \gamma$ while $\beta'<\beta\wedge[(\gamma-2)\beta]$ and $2 \beta+\beta'>1$.
	We first construct the solution on a small time interval $[0,\bar T]$, then later extend it to $[0,T]$.
	Since $|\X|_{\beta;[0,\bar T]}\le (\bar T+\bar T^2)^{\alpha- \beta}|\X|_{\alpha;[0,T]}$, we can choose $\bar T$ sufficiently small so that $|\X|_{\beta;[0,\bar T]}+\bar T\le 1$.
	For each $M>0$, we define $\mathbf{B}_{\bar T,M}$ as the collection of processes $(\eta,\eta')$ in $\mathbf{D}^{\beta,\beta'}_XL_m([0,\bar T],\W;W)$ such that $\eta_0=\xi, \eta'_0=\hat f_0(\xi,\slashed \xi)$, 
	\begin{equation}\label{con.ball1}
		\sup_{t\in[0,\bar T]}\|\eta'_t\|_\infty\vee\|\delta \eta'\|_{\beta';m}\le\||\hat f|_1\|_\infty
		\tand
		\|\delta \eta\|_{\beta;m}\vee\|\E_{\bigcdot} R^\eta\|_{\beta+\beta';m}\le 1\,.
	\end{equation}
	It is easy to see that for $\bar T$ sufficiently small and $M$ sufficiently large, the set $\mathbf{B}_{\bar T,M}$ contains the process $(\omega;t)\mapsto (\xi(\omega)+\hat f_0(\xi(\omega),\slashed \xi)\delta X_{0,t},\hat f_0(\xi(\omega),\slashed\xi))$, and hence, is non-empty.
	
	For each $(\eta,\eta')$ in $\mathbf{B}_{\bar T,M}$, define $(f^\eta,(f^\eta)')$ according to \eqref{def.feta} so that
	\begin{align*}
		Df^\eta_t(\omega,y)=D_1\hat f_t(\omega,y,\slashed \eta_t)
		\tand (Df^\eta)'_t(\omega,y)=D_2D_1\hat f_t(\omega,y,\slashed\eta_t)[\slashed \eta'_t]+D_1\hat f'_t(\omega,y,\slashed \eta_t).
	\end{align*}
	Similarly, define $b^\eta_t(\omega,y)=\hat b_t(\omega,y,\slashed \eta_t)$ and $\sigma^\eta_t(\omega,y)=\hat \sigma_t(\omega,y,\slashed \eta_t)$. 
	Since $D_2D_1\hat f=D_1D_2\hat f$, we have $(Df^\eta)'=D((f^\eta)')$.
	By \cref{prop.newcvec} and \cref{assume.regularity}, we see that $(f^\eta,(f^\eta)')$ and $(Df^\eta,(Df^\eta)')$ are stochastic  controlled vector fields in $\mathbf{D}^{\beta,\beta'}_XL_{m,\infty}\C^\gamma_b$ and $\mathbf{D}^{\beta,\beta'}_XL_{m,\infty}\C^{\gamma-1}_b$ respectively. 
	Furthermore, from the estimate in \cref{prop.newcvec} and \eqref{con.ball1}, we see that 
	\begin{multline}\label{est.unif.coef}
		\sup_{M}\sup_{(\eta,\eta')\in \mathbf{B}_{\bar T,M}}\Big(\||\sigma^\eta|_\infty\|_\infty +\||b^\eta|_\infty\|_\infty
		\\+ \|(f^\eta,(f^\eta)')\|_{\gamma;\infty}+\bk{(f^\eta,(f^\eta)')}_{X;\beta,\beta';m,\infty}+\bk{(Df^\eta,(Df^\eta)')}_{X;\beta,\beta';m,\infty}\Big)<\infty.
	\end{multline}
	
	Consider the rough stochastic differential equation
	\begin{equation}\label{eqn.Yeta}
		dY_t=b^\eta_t(Y_t)dt+\sigma^\eta_t(Y)dB_t+(f^\eta_t,(f^\eta_t)')(Y_t)d\X_t, \quad Y_0=\xi.
	\end{equation}
	Applying \cref{thm.fixpoint} and taking into account \eqref{est.unif.coef}, we see that there is a unique  solution $Y^\eta$ to \eqref{eqn.Yeta} which is $L_{m,\infty}$-integrable.
	Additionally, from  \cref{prop.apri} and \eqref{est.unif.coef}, we have 
	\begin{align}\label{est.unif.Ymu}
		\sup_{M}\sup_{(\eta,\eta')\in \mathbf{B}_{\bar T,M} }\Big(\|\delta Y^\eta\|_{\alpha;m,\infty;[0,\bar T]}+\|\E_\bigcdot R^{Y^\eta}\|_{\alpha+\beta;\infty;[0,\bar T]}\Big)<\infty.
	\end{align}
	Since $\alpha>\beta> \beta'$, the above estimate implies that $(Y^\eta,f^\eta(Y^\eta))$ belongs to $\mathbf{B}_{\bar T,M}$ for $\bar T$ sufficiently small and $M$ sufficiently large.
	
	Define $\Psi(\eta,\eta')=(Y^\eta,f^\eta(Y^\eta))$. The previous argument shows that $\Psi$ maps $\mathbf{B}_{\bar T,M}$ into $\mathbf{B}_{\bar T,M}\cap \mathbf{D}^{2 \beta}_XL_{m,\infty}$ when $\bar T$ is sufficiently small and $M$ is sufficiently large. We show that $\Psi$ is a contraction with respect to the metric \( (Y,Y';\bar Y,\bar Y')\mapsto\|Y,Y';\bar Y,\bar Y'\|_{X;\beta,\beta';m} \) introduced in \eqref{def.scrp.metricd}.
	
	For $(\eta,\eta'),(\zeta,\zeta')$ in $\mathbf{B}_{\bar T,M}$,
	it follows from \cref{thm.stability_precise} and \eqref{est.unif.coef} that
	\begin{multline}
		\label{preliminary_est_Y_diff}
		\|\delta Y^\eta- \delta Y^\zeta\|_{\alpha;m}+\|\delta f^\eta(Y^\eta)- \delta f^\zeta(Y^\zeta)\|_{\beta;m}+\|\E_\bigcdot R^{Y^\eta}-\E_\bigcdot R^{Y^\zeta}\|_{\alpha+\beta;m}
		\\
		\lesssim \sup\nolimits_{t\in[0,\bar T]}\||b^\eta_t-b^\zeta_t|_\infty\|_\infty
		+\sup\nolimits_{t\in[0,\bar T]}\||\sigma^\eta_t-\sigma^\zeta_t|_\infty\|_\infty
		\\
		\|(f^\eta-f^\zeta;(f^\eta)'-(f^\zeta)')\|_{\gamma-1;m}+\bk{f^\eta,(f^\eta)';f^\zeta,(f^\zeta)'}_{X;\beta,\beta';m}.
	\end{multline}
	We then estimate each term on the right-hand side above.
	By \cref{assume.regularity}\ref{item.reg.bs}, we have that
	\begin{align*}
		|b^\eta_t-b^\zeta_t|_\infty+|\sigma^\eta_t-\sigma^\zeta_t|_\infty \le (|\hat b|_1+|\hat \sigma|_1)\|\eta_t- \zeta_t\|_m.
	\end{align*}
	Moreover, \cref{cor.expo} yields that
	\begin{align*}
		\|(f^\eta-f^\zeta,(f^\eta)'-(f^\zeta)')\|_{\gamma-1;m}+\bk{f^\eta,(f^\eta)';f^\zeta,(f^\zeta)'}_{X;\beta,\beta';m}\lesssim \|\eta,\eta';\zeta,\zeta'\|_{X;\beta,\beta';m}.
	\end{align*}
	Plugging these estimates in \eqref{preliminary_est_Y_diff}, we have
	\begin{multline}\label{equ:estimateYmunu}
		\|\delta Y^\eta- \delta Y^\zeta\|_{\alpha;m}+\|\delta f^\eta(Y^\eta)- \delta f^\zeta(Y^\zeta)\|_{\beta;m}+\|\E_\bigcdot R^{Y^\eta}-\E_\bigcdot R^{Y^\zeta}\|_{\alpha+\beta;m}
		\\\le C \|\eta,\eta';\zeta,\zeta'\|_{X;\beta,\beta';m}
	\end{multline}
	for some finite constant $C$. 
	Since  $\alpha>\beta> \beta'$, there is a constant $o_{\bar T}(1)$ such that $\lim_{\bar T\downarrow0}o_{\bar T}(1)=0$ and
	\begin{multline*}
		\|\Psi(\eta,\eta');\Psi(\zeta,\zeta')\|_{X;\beta,\beta';m}
		\\
		\le o_{\bar T}(1)\Big(\|\delta Y^\eta- \delta Y^\zeta\|_{\alpha;m}+\|\delta f^\eta(Y^\eta)- \delta f^\zeta(Y^\zeta)\|_{\beta;m}+\|\E_\bigcdot R^{Y^\eta}-\E_\bigcdot R^{Y^\zeta}\|_{\alpha+\beta;m}\Big).
	\end{multline*}
	Hence, we can choose $\bar T$ sufficiently small so that
	\begin{align*}
        \|\Psi(\eta,\eta');\Psi(\zeta,\zeta')\|_{X;\beta,\beta';m}\le\frac12\|\eta,\eta';\zeta,\zeta'\|_{X;\beta,\beta';m}
    \end{align*}
	for every $(\eta,\eta'),(\zeta,\zeta')$ in $\mathbf{B}_{\bar T,M}$. In other words, $\Psi$ is a contraction on the (complete) metric space $(\mathbf{B}_{\bar T,M},\|\cdot;\cdot\|_{X;\beta,\beta';m})$.
	
	Let $Y$ be the fixed point of $\Psi$.  This means that $Y$ is an $L_{m,\infty}$-integrable solution to \eqref{eqn.Yeta} with $\eta=Y$, $\eta'=f^\eta(Y)=\hat f(Y,\slashed Y)=Y'$. Therefore, $Y$  is a $L_{m,\infty}$-integrable solution to \eqref{eqn.MVrsde}  (and equivalently to \eqref{eqn.MV}) on $[0,\bar T]$. 
	
	Since the smallness of $\bar T$ depends only on the controlling quantities of the coefficients, not on the starting position, the previous argument can be repeated to construct a solution on arbitrary finite time intervals.
	Sample path continuity is evident because $Y^\eta$ is a.s.\ continuous for any $(\eta,\eta')$ in $\mathbf{B}_{\bar T,M}$. 
	
	Lastly, we show uniqueness among $L_{m,\infty}$-solutions. Let $\bar Y$ be a $L_{m,\infty}$-integrable solution to \eqref{eqn.MV}. Then $\bar Y$ is a $L_{m,\infty}$-integrable solution to \eqref{eqn.Yeta} with $(\eta,\eta')=(\bar Y,\hat f(\bar Y,{\slashed {\bar Y}}))$. 
	The estimate \eqref{est.defJMV} and boundedness of the coefficients imply that
	\begin{align*}
		\|\delta \bar Y\|_{\alpha;m,\infty;[0,T]}+\|\E_\bigcdot R^{\bar Y}\|_{2 \alpha;\infty}<\infty.
	\end{align*}
	Using the regularity of $f$, we have
	\begin{align*}
		|\hat f_t(\bar Y_t,\slashed{\bar Y_t})-\hat f_s(\bar Y_s,\slashed{\bar Y_s})|\lesssim \sup_{(y,\xi)\in W\times  L_q(W)}|\delta \hat f_{s,t}(y,\xi)|+ |\delta Y_{s,t}|+\|\delta Y_{s,t}\|_m
	\end{align*}
	which implies that $\|\delta \hat f(\bar Y,\slashed{\bar Y})\|_{\alpha;m,\infty}$ is finite.
	We obtain from the above that $(\eta,\eta')= (\bar Y,\hat f(\bar Y,\slashed{\bar Y}))$ satisfies \eqref{con.eta1new}. Applying \cref{prop.newcvec}, we see that
	\begin{equation*}
		|\sigma^{\bar Y}|_\infty+|b^{\bar Y}|_\infty
		+ \|(f^{\bar Y},(f^{\bar Y})')\|_{\gamma;\infty}+\bk{(f^{\bar Y},(f^{\bar Y})')}_{X;\alpha,\alpha;m,\infty}+\bk{(Df^{\bar Y},(Df^{\bar Y})')}_{X;\alpha,\alpha;m,\infty}
	\end{equation*}
    is finite.
	In view of \cref{thm.fixpoint}, $\bar Y$ is also a $L_{m,\infty}$-solution to \eqref{eqn.Yeta}. In particular, $(\bar Y,\hat f(\bar Y,\slashed{\bar Y}))$ belongs to the domain of $\Psi$ and one has $(\bar Y,\hat f(\bar Y,\slashed{\bar Y}))=\Psi(\bar Y,\hat f(\bar Y,\slashed{\bar Y}))$. In other words, $(\bar Y,\hat f(\bar Y,\slashed{\bar Y}))$ is another fixed-point of $\Psi$, which is unique by the previous consideration.
\end{proof}

\subsection{Continuous dependence with respect to data} 
\label{sec:stability_proof_of_theorem_thm.mkv.stability}
We now address the proof of \cref{thm.MKV.Stability}.
Since the left-hand side in \eqref{est.mkv.stability} does not depend on the parameter \( \beta' \), we assume without loss of generality that \( \beta'\le(\gamma-2)\beta \) and \(\beta' < \beta\).
We denote by \( \varDelta \) the right-hand side of \eqref{est.mkv.stability}, that is 
\[ \begin{aligned}\varDelta= \|\xi-\bar \xi\|_p+\rho_{\alpha,\beta}(\X,\bar \X)
	+\sup_{t\in[0,T]}\||\sigma_t-\bar \sigma_t|_\infty\|_m
	+\sup_{t\in[0,T]}\||b_t-\bar b_t|_\infty\|_m
	\\
	+\|(\hat f-\hat {\bar f},\hat f'-\hat{\bar f}')\|_{\gamma-1;m}+\bk{\hat f,\hat f';\hat{\bar f},\hat{\bar f}'}_{X,\bar X;\beta,\beta';m}\,.
\end{aligned}\]

From \cref{def.MVsoln}, we have
\begin{align*}
	\|\|\delta Y_{s,t}|\cff_s\|_m\|_\infty\lesssim(t-s)^\alpha 
	,\quad\|\E_sR^Y_{s,t}\|_\infty\lesssim(t-s)^{2 \alpha}.
\end{align*}
Using the fact that $(\hat f,\hat f')\in \mathbf{D}^{\beta,\beta}_XL_{m,\infty}\C^\gamma_{b,p,q}$, we obtain that
\begin{align*}
	\|\|\delta f(Y,\slashed Y) _{s,t}|\cff_s\|_m\|_\infty\lesssim(t-s)^{\alpha\wedge \beta}.
\end{align*}
These estimates show that $(Y,\hat{f}(Y,\slashed Y))$ belongs to $\mathbf{D}^{\beta,\beta}_XL_{m,\infty}$. Similarly, one sees that $(\bar Y,\hat{\bar f}(\bar Y,\slashed {\bar Y}))$ belongs to $\mathbf{D}^{\beta,\beta}_{\bar X}L_{m,\infty}$.
We apply \cref{cor.expo} (with \((\kappa,\kappa')=(\beta,\beta')\) therein)  to get that 
\begin{equation}
	\label{fourth_bound}
	\|(f^Y-\bar f^{\bar Y},(f^Y)'-(\bar f^{\bar Y})')\|_{\gamma-1;m;[u,v]}
	\les \varDelta + \sup_{t\in[u,v]}\|Y_t-\bar Y_t\|_{p;[u,v]} 
\end{equation}
and		
\begin{equation}
	\label{fifth_bound}
	\bk{f^Y,(f^Y)';\bar f^{\bar Y},(\bar f^{\bar Y})'}_{X,\bar X;\beta,\beta';m;[u,v]}
	\les \varDelta + \|Y,Y';\bar Y,\bar Y'\|_{X,\bar X;\beta,\beta';p;[u,v]}\,.
\end{equation}

Furthermore, \cref{prop.newcvec} ensures that
\begin{itemize}
	\item $(f^Y,(f^Y)')$ belongs to $\mathbf{D}^{\beta,\beta}_XL_{m,\infty}\C^\gamma_b$, 
	\item $(Df^Y,D(f^Y)')$ belongs to $\mathbf{D}^{\beta,\beta'}_XL_{m,\infty}\C^{\gamma-1}_b$,
	\item $(\bar f^{\bar Y},(\bar f^{\bar Y})')$ belongs to $\mathbf{D}^{\beta,\beta'}_{\bar X}L_{m,\infty}\C^{\gamma-1}_b$.
\end{itemize}
We may apply \cref{thm.stability_precise} and find that for every $(u,v)\in \Delta(0,T)$,
\begin{align}
	&\|\sup_{t\in[u,v]}|\delta Y_{u,t}-\delta\bar Y_{u,t}|\|_m+ \|\delta Y- \delta\bar Y\|_{\alpha;m;[u,v]}+\|\delta f^{Y}(Y)-\delta\bar f^{\bar Y}(\bar Y)\|_{\beta;m;[u,v]} 
	\nonumber\\&+ \|\E_{\bigcdot} R^Y-\E_{\bigcdot}\bar R^{\bar Y}\|_{\alpha+ \beta;m;[u,v]}
	\lesssim \varDelta 
	+\sup_{t\in[u,v]}\||\sigma^Y_t-\bar \sigma^{\bar Y}_t|_{\infty}\|_m+\sup_{t\in[u,v]}\||b^Y_t-\bar b^{\bar Y}_t|_{\infty}\|_m
	\nonumber\\&\quad
	+\|(f^Y-\bar f^{\bar Y},(f^Y)'-(\bar f^{\bar Y})')\|_{\gamma-1;m;[u,v]}+\bk{f^Y,(f^Y)';\bar f^{\bar Y},(\bar f^{\bar Y})'}_{X,\bar X;\beta,\beta';m;[u,v]}\,.
	\label{est.YbarY}
\end{align}
Using regularity assumptions for \(b,\sigma\), it is easy to see that
\begin{align}\label{tmp.driftandsig}
	\sup_{t\in[u,v]}\||\sigma^Y_t-\bar \sigma^{\bar Y}_t|_{\infty}\|_m+\sup_{t\in[u,v]}\||b^Y_t-\bar b^{\bar Y}_t|_{\infty}\|_m\lesssim \varDelta+\sup_{t\in[u,v]}\|Y_t-\bar Y_t\|_p
\end{align}
while the last two terms in \eqref{est.YbarY} are treated in \eqref{fourth_bound} and \eqref{fifth_bound} respectively.
\begin{proof}[Conclusion]
	Combining  \eqref{fourth_bound}-\eqref{tmp.driftandsig}, we obtain the relation
	\begin{equation}
		\label{pre_gronwall}
		\|\sup_{t\in[u,v]}|\delta Y_{u,t}-\delta\bar Y_{u,t}|\|_m+\|Y,Y';\bar Y,\bar Y'\|_{X,\bar X;\alpha,\beta;m;[u,v]}
		\les \varDelta + \|Y,Y';\bar Y,\bar Y'\|_{X,\bar X;\beta,\beta';p;[u,v]}.
	\end{equation}
	The conclusion now follows from a standard argument: 
	since the above estimate holds for every finite time intervals $[u,v]$ and that $\beta'< \beta<\alpha$ and \(m\ge p\), it is standard to obtain that \[
	\|Y,Y';\bar Y,\bar Y'\|_{X,\bar X;\beta,\beta';p;[0,T]} \les  \varDelta.
	\]
	Plugging this bound in \eqref{pre_gronwall} shows the desired estimate.
	This finishes the proof of \cref{thm.MKV.Stability}.
\end{proof}


%
\appendix

\section{}
In what follows,  \( (\Omega,\mathcal G, \mathbb P) \) and \( (\Omega',\mathcal G',\mathbb P') \) are Polish probability spaces,  \( \Omega' \) is atomless.

\subsection{Measure-theoretic facts}
\label{sec:measure_theoretic}

If \( \lambda \) is the Lebesgue measure on \( [0,1 ]\), observe
that the product probability space
\[
\Omega\otimes[0,1]:=(\Omega\times[0,1],\mathcal G\otimes\Bor([0,1]),\P\otimes \lambda)
\]
is atomless.
Being Polish and atomless, the latter is, in fact, measure-theoretically isomorphic to \( \Omega' \). This is indeed a consequence of the Measurable Isomorphism Theorem (see e.g.\ \cite[Chap.~2]{bogachev2007measure} and the references therein).
Define \( \pi=\pi_1\circ T^{-1} \) where \( T\colon \Omega\otimes[0,1]\to \Omega' \) is such an isomorphism and \(\pi_1\colon \Omega\otimes [0,1]\to\Omega\), \( \pi_1(\omega,r)= \omega\). The map 
\[
\begin{aligned}
	\pi\colon (\Omega',\mathcal G',\P')
	\longrightarrow 
	(\Omega,\mathcal G,\P)
\end{aligned}
\]
is clearly measurable, moreover it is measure-preserving since \( \P'(\pi^{-1}(B))= \P'(T\circ \pi_1^{-1}(B))= \P\otimes\lambda(\pi_1^{-1}(B)) =\P(B)\) for each \( B\in \mathcal G \).
Consequently, if \( Z \) is a random variable on \( \Omega \), we can define
\begin{equation}
	\label{slashed_Z}
	\slashed Z(\omega'):= Z(\pi(\omega')),\quad \quad \omega'\in\Omega'
\end{equation}
and we see that the following properties hold (the proof is omitted).
\begin{lemma}
	\label{lem:measurability}
	Fix a sigma-algebra \( \mathcal F\subset \mathcal G\) and let 
	\[
	\slashed {\mathcal{F}}:=\pi^{-1}(\mathcal F) \,.
	\]
	For every \( \mathcal F \)-measurable random variable \( Z\) on \( \Omega \) with values in a separable Banach space \( \mathcal Y\), the relation \eqref{slashed_Z}
	defines an \( \slashed{\mathcal {F}} \)-measurable random variable \( \slashed{Z} \) on \(\Omega'\) whose law is identical to $Z$'s. The resulting map
	\begin{equation}
		\label{slashing}
		\slash\colon L_q(\mathcal F,\mathcal Y)\to L_q(\slashed{\mathcal F},\mathcal Y),\quad Z\mapsto \slashed Z
	\end{equation}
	is an isometry.
\end{lemma}

\subsection{Banach space-valued random variables}
\label{app:Lq}
Recall that \( (\Omega,\mathcal F,\mathbb P) \) is a Polish probability space.
Let \( \aby\) be a Banach space and take \( q\in [0,\infty] \). We denote by \( L_q=(L_q(\mathcal Y;\mathcal G),d_q)\) the vector space of random variables of order \( q \) equipped with the metric
\[
d_q(Z,\bar Z)= \|Z-\bar Z\|_{q} 
\]
(for \( q\ge1 \) this is clearly a distance; when \( q<1 \) use \( (x+y)^q\le x^q +y^q\), \( x,y\ge0 \) to prove triangle inequality).
We now collect a few other convenient properties for these spaces. 
\begin{lemma}
	\label{lem:embedding}
	Each \( L_q \), \( q\in [0,\infty] \) is a complete metric space, moreover
	\( d_0 \) topologizes convergence in probability, namely \( \lim Z_n\to Z\) in \( L_0 \) if and only if for any \( \epsilon>0 \), \( \lim\P(|Z_n-Z|\ge\epsilon)=0\).
	When \( q<\infty \), these are separable (hence Polish) vector spaces provided that \( \mathcal Y \) is separable.
	Finally,
	\begin{equation}
		\label{Lp_subset_Lq}
		L_p\hookrightarrow L_q,\quad \quad \forall 0\le q\le p\le\infty\,,
	\end{equation}
	and the corresponding embedding is dense if \( \mathcal Y \) is separable.
\end{lemma}
\begin{proof}
	These are standard properties.
	For instance, the fact that \( d_0 \) topologizes convergence in probability easily follows by Markov inequality.
	Separability is a consequence of the separability of \( \Omega \) (as it is Polish) and \cite[p.~111]{vakhania1987probability}. 
	The continuous embeddings in \eqref{Lp_subset_Lq} are easily obtained from H\"older inequality. As for density, it follows by the case \( \mathcal Y=\R\) and a diagonal extraction procedure, observing that \( \|\xi_n-\xi\|_q\to 0 \) where \( \xi_n:=(-n)\vee (\xi\wedge n) \in L_\infty(\R)\).
\end{proof}

\subsection{Conditional expectation}

Fix a Banach space \( \aby \), a random variable \( Z\colon \Omega\to \mathcal Y \) and let \( \mathcal F\subset \mathcal G \) be a \( \sigma \)-algebra. Classically, the conditional expectation \( \E_{\mathcal F} Z =\E(Z|\mathcal F) \) is defined through the property
\begin{equation}
	\label{conditional_banach}
	\int _A \E(Z|\mathcal F)(\omega)\P(d\omega) = \int _A Z(\omega)\P(d\omega),
	\quad \quad \text{for all }A\in \mathcal F\,.
\end{equation}
Recall that when \( Z \in L_1(\mathcal G;\aby)\), the conditional expectation indeed exists and \eqref{conditional_banach} determines \( \E(Z|\mathcal F)\in L_1(\mathcal F;\aby) \) uniquely \cite[Chap.~V.1]{diestel1977vector}. Moreover, the resulting map \( \E(\cdot |\mathcal F)\colon L_1(\mathcal G;\aby)\to L_1(\mathcal F,\aby) \) is linear
and its operator norm does not exceed \( 1 \) (contraction property).

We now record the following fact about the possibility of interchanging linear operators and conditional expectations in some cases.
The result should be well-known in principle, but it is not much effort to provide a proof, which we do. For related -- though slightly different -- statements, we refer the reader to the classical monograph \cite{MR1102015} (specifically: Chapter 2 and the references therein).
\begin{lemma}
	\label{lem:operator}
	Let \( \mathcal X,\mathcal Y \) be real Banach spaces and fix a \( \sigma \)-algebra \( \mathcal F\subset \mathcal G \).
	Let \( \{T_\omega\} \) be a random family of continuous linear operators from \( \mathcal X\to \mathcal Y \) such that
	 \[
	 \begin{aligned}
	T_{\bigcdot}\colon L_1(\mathcal G;\abx)
	\rightarrow L_1(\mathcal G;\aby),\quad \quad 
	 \eta 
	 \mapsto 
	 \Big(\omega\mapsto T_\omega[\eta(\omega)] \Big)
	 \end{aligned}
	 \]
	 is well-defined, bounded and such that
	 \begin{equation}
	 	\label{condition_T}
	 	 T_{\bigcdot}(L_1(\mathcal F;\abx))\subset L_1(\mathcal F;\aby) .
	 \end{equation}
	Then, conditional expectation and \( T \) may be swapped as follows:
	for any \( \eta\in L_1(\abx;\mathcal G) \),
	\[
	\E(T[\eta]|\mathcal F) = T[\E(\eta|\mathcal F)],\quad \P\otimes\P'\text{-a.s.} 
	\]
	
	Consequently, if \( Z\in L_\infty(\mathcal F;\mathcal C_b^{\gamma}(\abx;\aby)) \) for some \( \gamma>1 \) and
	 \( T_\omega= DZ(\omega) \), we find that for any \( \eta\in L_1(\abx;\mathcal G) \)
	 \[
	 \E(DZ[\eta]|\mathcal F) = DZ[\E(\eta|\mathcal F)],\quad \P\otimes\P'\text{-a.s.} 
	 \]
\end{lemma}
\begin{proof}
We suppose that \( \mathcal Y=\R \)
(the general case follows by Hahn--Banach Theorem). 
Let us assume first that \( T \) is deterministic, namely there exists \( \tau\in \abx^*=\mathcal L(\abx,\R) \) such that  \( T_\omega=\tau \), \( \P \)-a.s.
	Now, linear functionals and integrals commute (by definition of vector-valued integration): we thus infer that
	for every \( A\in \mathcal F \):
	\[
	\begin{aligned}
	\int_{A}\langle\tau,\E_{\mathcal F}\eta\rangle(\omega)\P(d\omega)
	&=
	\langle \tau,\int _A \E_{\mathcal F}\eta (\omega)\P(d\omega)\rangle
	\\&= \langle\tau,\int _A \eta(\omega)\P(d\omega)\rangle
	\\&= \int _A \langle\tau,\eta(\omega)\rangle\P(d\omega),
	\end{aligned}
	\]
	where we used \eqref{conditional_banach} at the second line.
We find indeed that
\[
\E_{\mathcal F}\langle\tau,\eta\rangle= \langle\tau,\E_{\mathcal F}\eta\rangle,\quad 
\P\text{-a.s.},
\]
which proves our claim in the deterministic case. 

Next, if we suppose that \( T_\omega=\sum_{i=1}^n\mathbf 1_{A_i}(\omega)\tau_i \) is simple, we find indeed that
\[
\E_{\mathcal F}T[\eta]=\sum_{i=1}^n\mathbf 1_{A_i}\E_{\mathcal F}\langle \tau_i,\eta\rangle
=\sum_{i=1}^n\mathbf 1_{A_i}\langle \tau_i,\E_{\mathcal F}\eta\rangle
=T[\E_{\mathcal F}\eta]
\]
by the deterministic case. Finally, if \( T^n,n\ge1 \) is a sequence of such simple elements converging \( \P \)-a.s.\ to \( T \) in \( \mathcal X^* \), we obtain
\[
\E_{\mathcal F}T[\eta]=\lim_{n\to\infty}\E_{\mathcal F}T^n[\eta]=\lim_{n\to\infty} T^n[\E_{\mathcal F}\eta]=T[\E_{\mathcal F}\eta]
\]
 thanks to the previous case and contraction property.
\end{proof}

Next, we record an observation on conditional expectation in product space. Again, while the property is fundamentally classical, finding a reference turns out to be challenging, so we do provide a proof.
\begin{lemma}
	\label{lem:prod}
	Let $\eta$ be a \( \aby \)-valued integrable random variable on $\Omega\otimes\Omega'$ and fix two sigma-algebras \( \mathcal F\subset \mathcal G\), \( \mathcal F'\subset \mathcal G' \).
	For \( \mathbb P \)-a.e.\ $\omega$ (resp.\ for \( \mathbb P '\)-a.e.\ $\omega'$), the random variable 
	\[
	\Omega'\to \aby,\quad \omega'\mapsto\E\otimes\E'(\eta|\cff\otimes\cff')(\omega,\omega') 
	\]
	(resp.\ \( \Omega'\to \aby, \) \(\omega\mapsto\E\otimes\E'(\eta|\mathcal F\otimes\cff')(\omega,\omega')\))
	is $\cff'$-measurable (resp.\ \( \mathcal F \)-measurable) and it holds
	\[	
	(\E\otimes\E')(\eta|\cff\otimes\cff')=
	\E(\E'( \eta|\mathcal F')|\mathcal F)=
	\E'(\E( \eta|\mathcal F)|\mathcal F')
	\,,
	\]
	\( \P\otimes\P' \)-almost-surely.
\end{lemma} 
\begin{proof}
	Let $\zeta=\E\otimes\E'(\eta|\cgg\otimes\cff')$.
	For every $A\in\cgg$ and $A'\in\cff'$, Fubini Theorem yields
	\[
	\E\otimes\E'(\zeta\mathds{1}_{A}\mathds{1}{A'})=\E\E' [\E'( \eta|\cff')\mathds{1}_{A}\mathds{1}_{A'}]\,.
	\]
	It follows that $\zeta=\E'( \eta|\cff')$ $\P$-a.s.,  $\P'$-a.s and hence also $\P\otimes\P'$-a.s. Analogously, we obtain that
	$\E\otimes\E'(\eta|\cff\otimes\cgg')=\E( \eta|\cff)$ $\P\otimes\P'$-a.s. 
	Using tower property of conditioning and that $\cff\otimes\cff'=(\cgg\otimes\cff')\wedge(\cff\otimes\cgg')$,  we have
	\begin{align*}
		(\E\otimes\E')(\eta|\cff\otimes\cff')=\E\otimes\E'_{\cgg\otimes\cff'}[\E\otimes\E'_{\cff\otimes\cgg'} \eta]=\E_{\cff}\E'_{\cff'} \eta \quad\P\otimes\P'\text{-a.s.}
	\end{align*}
	This proves our claim.		
\end{proof}

\bibliographystyle{alpha}
\bibliography{processes_martingale}

\begin{bibdiv}
\begin{biblist}

\bib{MR3420480}{article}{
      author={Bailleul, Isma\"{e}l},
       title={Flows driven by rough paths},
        date={2015},
        ISSN={0213-2230},
     journal={Rev. Mat. Iberoam.},
      volume={31},
      number={3},
       pages={901\ndash 934},
         url={https://doi.org/10.4171/RMI/858},
      review={\MR{3420480}},
}

\bib{MR4073682}{article}{
      author={Bailleul, Isma\"{e}l},
      author={Catellier, R\'{e}mi},
      author={Delarue, Fran\c{c}ois},
       title={Solving mean field rough differential equations},
        date={2020},
     journal={Electron. J. Probab.},
      volume={25},
       pages={Paper No. 21, 51},
         url={https://doi.org/10.1214/19-ejp409},
      review={\MR{4073682}},
}

\bib{bailleul2020propagation}{article}{
      author={Bailleul, Ismaël},
      author={Catellier, Rémi},
      author={Delarue, François},
       title={{Propagation of chaos for mean field rough differential
  equations}},
        date={2021},
     journal={The Annals of Probability},
      volume={49},
      number={2},
       pages={944 \ndash  996},
         url={https://doi.org/10.1214/20-AOP1465},
}

\bib{bogachev2007measure}{book}{
      author={Bogachev, Vladimir~Igorevich},
      author={Ruas, Maria Aparecida~Soares},
       title={Measure theory},
   publisher={Springer},
        date={2007},
      volume={2},
}

\bib{carmona2015forward}{article}{
      author={Carmona, Ren{\'e}},
      author={Delarue, Fran{\c{c}}ois},
       title={Forward--backward stochastic differential equations and
  controlled mckean--vlasov dynamics},
        date={2015},
     journal={The Annals of Probability},
      volume={43},
      number={5},
       pages={2647\ndash 2700},
}

\bib{MR3752669}{book}{
      author={Carmona, Ren\'{e}},
      author={Delarue, Fran\c{c}ois},
       title={Probabilistic theory of mean field games with applications. {I}},
      series={Probability Theory and Stochastic Modelling},
   publisher={Springer, Cham},
        date={2018},
      volume={83},
        ISBN={978-3-319-56437-1; 978-3-319-58920-6},
        note={Mean field FBSDEs, control, and games},
      review={\MR{3752669}},
}

\bib{MR3753660}{book}{
      author={Carmona, Ren\'{e}},
      author={Delarue, Fran\c{c}ois},
       title={Probabilistic theory of mean field games with applications.
  {II}},
      series={Probability Theory and Stochastic Modelling},
   publisher={Springer, Cham},
        date={2018},
      volume={84},
        ISBN={978-3-319-56435-7; 978-3-319-56436-4},
        note={Mean field games with common noise and master equations},
      review={\MR{3753660}},
}

\bib{carmona2018probabilistic}{book}{
      author={Carmona, Ren{\'e}},
      author={Delarue, Fran{\c{c}}ois},
       title={Probabilistic theory of mean field games with applications ii:
  Mean field games with common noise and master equations},
   publisher={Springer},
        date={2018},
      volume={84},
}

\bib{CF16}{article}{
      author={Coghi, Michele},
      author={Flandoli, Franco},
       title={Propagation of chaos for interacting particles subject to
  environmental noise},
        date={2016},
        ISSN={1050-5164},
     journal={Ann. Appl. Probab.},
      volume={26},
      number={3},
       pages={1407\ndash 1442},
         url={https://doi.org/10.1214/15-AAP1120},
      review={\MR{3513594}},
}

\bib{coghi2019stochastic}{article}{
      author={Coghi, Michele},
      author={Gess, Benjamin},
       title={Stochastic nonlinear {F}okker--{P}lanck equations},
        date={2019},
     journal={Nonlinear Analysis},
      volume={187},
       pages={259\ndash 278},
}

\bib{MR3299600}{article}{
      author={Cass, Thomas},
      author={Lyons, Terry},
       title={Evolving communities with individual preferences},
        date={2015},
        ISSN={0024-6115},
     journal={Proc. Lond. Math. Soc. (3)},
      volume={110},
      number={1},
       pages={83\ndash 107},
         url={https://doi.org/10.1112/plms/pdu040},
      review={\MR{3299600}},
}

\bib{coghi2019rough}{article}{
      author={Coghi, Michele},
      author={Nilssen, Torstein},
       title={Rough nonlocal diffusions},
        date={2021-11},
     journal={Stochastic Processes and their Applications},
      volume={141},
       pages={1\ndash 56},
         url={https://doi.org/10.1016/j.spa.2021.07.002},
}

\bib{diestel1977vector}{book}{
      author={Diestel, J.},
      author={J.J.~Uhl, Jr.},
       title={Vector measures},
   publisher={American Mathematical Society},
        date={1977},
}

\bib{delarue2021probabilistic}{article}{
      author={Delarue, Francois},
      author={Salkeld, William},
       title={Probabilistic rough paths {I} {L}ions trees and coupled {H}opf
  algebras},
        date={2021},
     journal={arXiv preprint arXiv:2106.09801},
}

\bib{delarue2022probabilisticroughpathsii}{article}{
      author={Delarue, François},
      author={Salkeld, William},
       title={Probabilistic rough paths {I}{I} {L}ions-{T}aylor expansions and
  random controlled rough paths},
        date={2022},
      eprint={2203.01185},
         url={https://arxiv.org/abs/2203.01185},
}

\bib{FHL21}{article}{
      author={Friz, Peter},
      author={Hocquet, Antoine},
      author={L{\^e}, Khoa},
       title={Rough stochastic differential equations},
        date={2021},
     journal={arXiv preprint arXiv:2106.10340},
}

\bib{FLZ25}{incollection}{
      author={Friz, Peter~K.},
      author={L\^{e}, Khoa},
      author={Zhang, Huilin},
       title={Randomisation of rough stochastic differential equations},
        date={2025},
   booktitle={Springer {P}roceedings in {M}athematics and {S}tatistics.
  {S}tochastic {A}nalysis and {A}pplications 2025: In {H}onour of {T}erry
  {L}yons ({S}{A}{A}{T} 2025)},
      series={Springer Proceedings in Mathematics and Statistics},
   publisher={Springer},
         url={https://arxiv.org/abs/2503.06622},
        note={{T}o appear, also available ast arXiv:2503.06622},
}

\bib{friz2020rough}{article}{
      author={Friz, Peter~K.},
      author={Zorin-Kranich, Pavel},
       title={Rough semimartingales and p-variation estimates for martingale
  transforms},
        date={2023-03},
        ISSN={0091-1798},
     journal={The Annals of Probability},
      volume={51},
      number={2},
         url={http://dx.doi.org/10.1214/22-AOP1598},
}

\bib{MR4260494}{article}{
      author={Hammersley, William R.~P.},
      author={\v{S}i\v{s}ka, David},
      author={Szpruch, \L~ukasz},
       title={Mc{K}ean-{V}lasov {SDE}s under measure dependent {L}yapunov
  conditions},
        date={2021},
        ISSN={0246-0203},
     journal={Ann. Inst. Henri Poincar\'{e} Probab. Stat.},
      volume={57},
      number={2},
       pages={1032\ndash 1057},
         url={https://doi.org/10.1214/20-aihp1106},
      review={\MR{4260494}},
}

\bib{MR2277714}{incollection}{
      author={Jouini, Ely\`es},
      author={Schachermayer, Walter},
      author={Touzi, Nizar},
       title={Law invariant risk measures have the {F}atou property},
        date={2006},
   booktitle={Advances in mathematical economics. {V}ol. 9},
      series={Adv. Math. Econ.},
      volume={9},
   publisher={Springer, Tokyo},
       pages={49\ndash 71},
         url={https://doi.org/10.1007/4-431-34342-3_4},
      review={\MR{2277714}},
}

\bib{le2018stochastic}{article}{
      author={L\^{e}, Khoa},
       title={A stochastic sewing lemma and applications},
        date={2020},
     journal={Electron. J. Probab.},
      volume={25},
       pages={Paper No. 38, 55},
         url={https://doi.org/10.1214/20-ejp442},
      review={\MR{4089788}},
}

\bib{lacker2017limit}{article}{
      author={Lacker, Daniel},
       title={Limit theory for controlled {M}c{K}ean--{V}lasov dynamics},
        date={2017},
     journal={SIAM Journal on Control and Optimization},
      volume={55},
      number={3},
       pages={1641\ndash 1672},
}

\bib{LeSSL2}{article}{
      author={L{\^e}, Khoa},
       title={Stochastic sewing in {B}anach space},
        date={2022},
     journal={arXiv preprint arXiv:2105.09364},
}

\bib{MR1102015}{book}{
      author={Ledoux, Michel},
      author={Talagrand, Michel},
       title={Probability in {B}anach spaces},
      series={Ergebnisse der Mathematik und ihrer Grenzgebiete (3) [Results in
  Mathematics and Related Areas (3)]},
   publisher={Springer-Verlag, Berlin},
        date={1991},
      volume={23},
        ISBN={3-540-52013-9},
         url={https://doi.org/10.1007/978-3-642-20212-4},
        note={Isoperimetry and processes},
      review={\MR{1102015}},
}

\bib{MR1431299}{incollection}{
      author={M\'{e}l\'{e}ard, Sylvie},
       title={Asymptotic behaviour of some interacting particle systems;
  {M}c{K}ean-{V}lasov and {B}oltzmann models},
        date={1996},
   booktitle={Probabilistic models for nonlinear partial differential equations
  ({M}ontecatini {T}erme, 1995)},
      series={Lecture Notes in Math.},
      volume={1627},
   publisher={Springer, Berlin},
       pages={42\ndash 95},
         url={https://doi.org/10.1007/BFb0093177},
      review={\MR{1431299}},
}

\bib{scheutzow1987uniqueness}{article}{
      author={Scheutzow, Michael},
       title={Uniqueness and non-uniqueness of solutions of {V}lasov-{M}c{K}ean
  equations},
        date={1987},
     journal={Journal of the Australian Mathematical Society},
      volume={43},
      number={2},
       pages={246\ndash 256},
}

\bib{MR2651517}{article}{
      author={Svindland, Gregor},
       title={Continuity properties of law-invariant (quasi-)convex risk
  functions on {$L^\infty$}},
        date={2010},
        ISSN={1862-9679},
     journal={Math. Financ. Econ.},
      volume={3},
      number={1},
       pages={39\ndash 43},
         url={https://doi.org/10.1007/s11579-010-0026-x},
      review={\MR{2651517}},
}

\bib{MR1108185}{incollection}{
      author={Sznitman, Alain-Sol},
       title={Topics in propagation of chaos},
        date={1991},
   booktitle={\'{E}cole d'\'{E}t\'{e} de {P}robabilit\'{e}s de {S}aint-{F}lour
  {XIX}---1989},
      series={Lecture Notes in Math.},
      volume={1464},
   publisher={Springer, Berlin},
       pages={165\ndash 251},
         url={https://doi.org/10.1007/BFb0085169},
      review={\MR{1108185}},
}

\bib{vakhania1987probability}{book}{
      author={Vakhania, N.N.},
      author={Tarieladze, V.I.},
      author={Chobanyan, S.A.},
       title={Probability distributions on banach spaces},
   publisher={D. Reidel Publishing Company},
        date={1987},
}

\end{biblist}
\end{bibdiv}
\end{document}